\documentclass{article}

\date{2016.11.29}

\usepackage{amsmath,amssymb}	
\usepackage{amsthm}
\usepackage{amscd}

\usepackage{graphicx}		  

\usepackage{color}
\usepackage{braket}

\newtheorem{thm}{Theorem}
\newtheorem{lem}{Lemma}
\newtheorem{prop}{Proposition}
\newtheorem{cor}{Corollary}
\newtheorem{defn}{Definition}
\newtheorem{rem}{Remark}

\newtheorem{exam}{Example}

\def\C{\mathbb C}
\def\Q{\mathbb Q}
\def\Z{\mathbb Z}
\def\F{\mathcal F}
\def\G{\mathcal G}
\def\H{\mathcal H}
\def\sR{\mathcal R}
\def\sSP{\mathcal S \mathcal P}
\def\D{{\bf D}}
\def\O{\mathcal O}

\def\bF{{\bf F}}
\def\bSS{{\bf SS}}
\def\NC{{\rm NC}}

\def\op{\oplus}

\def\FGC{
\setlength\unitlength{0.20truemm}
\begin{picture}(30,30)(0,0)
\multiput(0,0)(0,30){2}{\line(1,0){30}}
\multiput(0,0)(30,0){2}{\line(0,1){30}}
\put(15,31){\line(0,-1){32}}
\put(31,15){\line(-1,0){32}}
\end{picture}
}

\def\FGP{
\setlength\unitlength{0.20truemm}
\begin{picture}(30,30)(0,0)
\multiput(0,0)(0,30){2}{\line(1,0){30}}
\multiput(0,0)(30,0){2}{\line(0,1){30}}
\put(14,31){\line(1,-1){18}}
\put(0,15){\line(1,-1){16}}
\put(16,-1){\line(-1,1){18}}
\end{picture}
}

\def\ln{\substack{\longleftarrow\\ n}}

\newcommand{\R}{\mathcal{R}}

\title{Construction of double Grothendieck polynomials \\
of classical types using IdCoxeter algebras}

\author{{Anatol N. Kirillov}{*} and {Hiroshi Naruse}{**}}

\begin{document}

\maketitle

\begin{center}
{Research Institute for Mathematical Sciences{*},\\
 RIMS, Kyoto University, Sakyo-ku,  606-8502 Japan\\
{\rm and }}
{The Kavli Institute for Physics and Mathematics
of the Universe{*} ,\\
IPMU, 5-1-5 Kashiwanoha, Kashiwa,  277-8583 Japan\\
{\rm and }}
{Department of Mathematics, National Research University Higher School of Economics,
117312, Moscow, Vavilova str. 7, Russia.}

{kirillov@kurims.kyoto-u.ac.jp}
\\[0.3cm]

{Graduate School of Education,
University of Yamanashi{**},\\
4-4-37 Takeda, Kofu, Yamanashi, 400-8510 Japan.}
{hnaruse@yamanashi.ac.jp}
\end{center}

\begin{abstract}
We construct double Grothendieck polynomials of classical types
which are essentially equivalent to but simpler than the polynomials
defined  by A.N.Kirillov in arXiv:1504.01469
and identify them with the polynomials defined by 
T.Ikeda and H.Naruse in Adv. Math.(2013)
for the case of maximal Grassmannian permutations.
We also give geometric interpretation of them in terms of
algebraic localization map 
and give explicit combinatorial formulas.
\end{abstract}

\section{Introduction}

  Let $G$ be a semisimple complex Lie group, $B \subset G$ be a 
  fixed
  Borel subgroup of $G$,
$T \subset B$ be a maximal torus in $B,$ ${\cal {F}}:=G/B$  and 
$W:=N_{G}(T)/T$ be the corresponding flag variety and the Weyl group. Let $\ell$
 be the rank of $G$.
According to the famous Borel's theorem \cite{Bo}, the cohomology ring
$H^{*}(G/B, {\mathbb{Q}})$  is isomorphic to the quotient 
 ${\mathbb {Q}}[z_1,\ldots,z_{\ell}]/ J_{\ell},$ 
where $z_i:= c_1(L_i) \in H^{2}(G/B,{\mathbb {Q}}),$ $i=1,\ldots,\ell,$ 
and  $c_1(L_i)$ denotes the first Chern class of the standard line bundle 
$L_i$ corresponding to the $i$-th fundamental weight $\omega_i$ 
over the complete flag variety ${\cal {F}}=G/B$ in question,
$J_{\ell}$ stands for the ideal 
generated by the fundamental invariants associated with  the 
Weyl group $W$ . 

To our best knowledge the first systematic and complete treatment of the 
 Schubert Calculus  has been done by I. N. Bernstein, I. M. Gelfand 
 and S. I. Gelfand  
 \cite{BGG} and independently, by  M. Demazure \cite{D} in the beginning of 
70's of the last century. A Schubert polynomial $ {{S}}_w (Z_{\ell})$, 
with $\ell=
{\rm rk}(G),$
$Z_\ell=(z_1,z_2,\ldots,z_\ell)$,  corresponding to an element $w$ of the Weyl group $W,$ 
by definition 
is a polynomial which expresses the Poincar\'e dual class 
$[X_{w_0 w}]\in H^{*}(G/B)$, where $w_0$ is the longest element 
in $W$, of the homology 
class of the  Schubert variety $X_w:=\overline{BwB/B} \subset G/B$ in terms of the 
Borel generators $z_i, 1 \le i \le \ell,$ in the cohomology ring of the 
flag variety ${\cal{F}}$.  Therefore by the very definition,  
a Schubert polynomial ${{{S}}_w}(Z_\ell)$ is defined only 
modulo the ideal $J_{\ell}.$

Hence it is an interesting problem to ask if there exists the 
``natural representative'' of a Schubert polynomial ${{{S}}_w}(Z_{\ell})$ in the ring ${\mathbb{Q}}[z_1,\ldots,z_\ell]$
 with ``nice'' 
combinatorial, algebraic and geometric properties.

For the type $A_{n-1}$ flag varieties,  
A. Lascoux and M.-P. Sch\"{u}tzenberger \cite{LS}
constructed a family of 
Schubert polynomials\footnote{
We refer the reader to nicely written book
\cite{Mac} for comprehensive exposition of the Schubert polynomials.}
${\mathfrak {S}}_w(X_{n})\in \Z[X_{n}]$
with
$w \in {{ S}}_{n}$ 
where $X_{n}=(x_1,x_2,\ldots,x_{n})$
are indeterminates, and $S_{n}$ is the symmetric group on the set of 
$n$ letters 
$\{1,2,\ldots, n\}$.
We will write the transposition $s_i=(i,i+1)$. Then 
$S_{n}$ is a Coxeter group with distinguished  set 
$I=\{s_1,s_2,\ldots,s_{n-1}\}$ of generators.
We list some of nice properties of the Schubert polynomials 
${\mathfrak {S}}_w(X_{n})$
according to \cite{FK2}.

\setlength{\itemsep}{20mm}
\begin{enumerate}

\item[(0)] ${\mathfrak {S}}_w(X_{n})$ is homogeneous of degree $\ell(w)$, ${\mathfrak {S}}_e(X_{n})=1$.
\\

\item[(1)] (Compatibility conditions)
$$\partial_{i}^{(x)}{\mathfrak{S}}_{w}(X_{n})= \begin{cases}
{\mathfrak{S}}_{ws_{i}}(X_{n})& \text{if $l(ws_{i})=\ell(w)-1$}, \\
0 & \text{otherwise} 
\end{cases}
$$
where $\partial_{i}^{(x)} f=\displaystyle\frac{f-s_i(f)}{x_i-x_{i+1}}$
 is the divided difference operator with respect to $x_i$ and $x_{i+1}$.

\item[(2)] 
the   structural constants for the multiplication of 
Schubert polynomials 
${\mathfrak {S}}_w(X_{n})$, $w \in {{S}}_{n},$ coincide with the 
triple intersection numbers of Schubert varieties,\\

\item[(3)] ${\mathfrak {S}}_w(X_{n})$ has 
nonnegative integer coefficients,
\\

\item[$(4_{w}), (4_{s})$] ${\mathfrak {S}}_w(X_{n})$ is weakly and strongly stable i.e.
for all $m>n$, we have
$$
 {\mathfrak {S}}_w(X_{m})={\mathfrak {S}}_w(X_{n}), 
\text{ where }
 w\in S_{n}\subset S_{m},$$
see Definition 8 in Section 5 below.
\\

\end{enumerate}

A new approach to the theory of type $A$ Schubert polynomials which is based 
on the study of the type $A$ nilCoxeter algebras, has been initiated by 
S. Fomin and R. Stanley \cite{FS}. The basic idea of that approach is to consider and 
study the generating function of all Schubert polynomials 
simultaneously, namely, to treat the following generating function
$${\mathfrak {S}}(X_{n}) = \sum_{w\in S_n} {\mathfrak {S}}_ w(X_{n}) u_w, $$
where $u_w$ denotes the standard linear basis in the type $A$ nilCoxeter algebra $NC_n$ 
which is a $\Z$-algebra with generators $u_1,u_2,\ldots, u_{n-1}$ and relations
$$u_i^2=0 (1\leq i\leq n-1),
u_i u_j=u_j u_i (|i-j|>1),
u_i u_{i+1} u_i=u_{i+1} u_i u_{i+1}  (1\leq i\leq n-2).
$$
We define $u_w=u_{{i_1}}\cdots u_{{i_\ell}}$ when 
$w=s_{i_1}\cdots s_{i_\ell}\in S_n$ is a 
reduced expression
by the transpositions $s_i=(i,i+1)$.
 An unexpected and deep result discovered
 in \cite{FS} is that in the algebra 
${\NC}_n [x_1,\ldots,x_{n}]=\NC_n\otimes \Z [x_1,\ldots,x_{n}]$ the polynomial 
${\mathfrak {S}}(X_{n})$ is
completely factorizable in the product of linear factors. The basic tool to 
prove the factorizability property is the usage of the Yang--Baxter relation 
among the elements  $h_i(x)= 1+ x u_i $ in the  algebra $\NC_n[x,y],$ namely
\begin{equation}
(1+x u_i)(1+(x+y) u_{i+1})(1+y u_i)=(1+y u_{i+1})(1+ (x+y)u_{i})(1+x u_{i+1}).
\end{equation}
The main consequence of the Yang--Baxter relation $(1)$ is that the 
polynomials  
$A_k(x) = h_{n-1}(x) h_{n-2 }(x) \ldots h_k(x),$ commute, namely
$$ [ A_k(x), A_k(y)] =0.$$ 
It has been proved in  \cite{FK3} , \cite{FS} that 
$$ {\mathfrak {S}}(X_{n})= \sum_{w \in 
{{S}}_{n}}
{\mathfrak {S}}_{w}(X_{n})u_{w} = A_1(x_1) A_2(x_2) \cdots A_{n-1}(x_{n-1}).$$

The double Schubert polynomials
${\mathfrak {S}}_w(X_n, Y_n)$ of type $A$, which were originally defined by A. Lascoux in \cite{Lc},
are combinatorially defined as follows.
For the longest element $w_0=[n-1,n-2,\ldots ,1]\in S_n$,  it is defined by
$${\mathfrak {S}}_{w_0}(X_n, Y_n):=\prod_{i+j\leq n} (x_i+y_j).$$
\noindent
For general $w\in S_n$, it is define using divided difference operator as
$${\mathfrak {S}}_w(X_n, Y_n):=\partial_{w^{-1} w_0}^{(x)} {\mathfrak {S}}_{w_0}(X_n, Y_n).$$
Using nilCoxeter algebra $\NC_n$ the generating function 
$\mathfrak{S}(X_n,Y_n)=\displaystyle\sum_{w\in S_n}  {\mathfrak {S}}_w(X_n, Y_n) u_w$ 
of double Schubert polynomials
can be factored as follows.
$$\mathfrak{S}(X_n,Y_n)=A_{n-1}^{-1}(-y_{n-1}) A_{n-2}^{-1}(-y_{n-2}) \cdots A_{1}^{-1}(-y_{1})
A_1(x_1) A_2(x_2) \cdots A_{n-1}(x_{n-1}).
$$
Later it was noticed by R. Goldin \cite{Gol} that
the double Schubert polynomials represent torus equivariant Schubert classes,
cf. Theorem 2.4 in \cite{Gol}.
When $y_1=y_2=\cdots=y_n=0$, the double Schubert polynomial  
$\mathfrak{S}_{w}(X_n, Y_n)$ becomes 
the single Schubert polynomial $\mathfrak{S}_w(X_n)$.

Construction  of ``good'' representatives for the Schubert polynomials corresponding to the 
flag varieties  of classical types $B, C,D$ was initiated by S. Billey and M. Haiman \cite{BH} 
and independently by S. Fomin and A. N. Kirillov
\cite{FK2}. In \cite{FK2} the authors extended an algebro-combinatorial approach 
(i.e. using nilCoxeter algebra and Yang-Baxter equations)
to a definition and study extending the type $A$ Schubert polynomials
to the case of those of types $B$ and $C$.
But it also works for type $D$ as well.
The key tool in a construction of the 
aforementioned polynomials is a unitary exponential solution to the quantum 
Yang--Baxter equations
(\cite{STF})
  with values in the nilCoxeter algebras of types $B,C,D$ 
correspondingly. The exponential solution to the quantum Yang--Baxter equation 
associated with nilCoxeter algebra $\NC(W)$,
(which is a specialization $\beta=0$ of IdCoxeter algebra ${\rm Id}_\beta (W)$ in Definition 1) of Weyl group $W=W(X)$
of root system of type
  $X:=A_{n-1},B_n,C_n,D_n,$ allows to 
construct a family of elements $R_{i}(x) \in \NC(R)[x]$
with $i=1,\ldots,{\rm rk}(R)$ such 
that 
$$R_i(x)R_i(y)=R_i(y)R_i(x),i=1,\ldots,{\rm rk}(R).$$
The elements $R_i(x_1),\ldots,R_{i}(x_{\ell})$
with $i=1,\ldots,\ell:={\rm rk}(R)$ 
are 
building blocks in the construction of the generating function for all 
Schubert  polynomials corresponding to the flag variety associated with the 
root system $R$.

Now in order to ensure the compatibility conditions one needs to 
specify the action of simple transpositions of the corresponding Weyl group on 
the ring of polynomials $\Q[x_1,\ldots,x_{\ell}]$. In \cite{FK2} and \cite{Ki} 
the  authors have chosen the natural or  {\it standard} action of the Weyl group on the 
cohomology ring of the corresponding flag variety $G/B$. Namely,
$$\begin{array}{l}
s_i(x_i)=x_{i+1}, s_i(x_{i+1})=x_i, s_i(x_j)=x_j\text{ if } j\neq i,i+1\; ({\rm type}\; A),\\
s_{0}(x_1)=-x_1,s_{0}(x_i)=x_i \text{ if } i > 1\; ({\rm types}\; B,C),\\
s_{\hat{1}} (x_1)=-x_2,
s_{\hat{1}} (x_2)=-x_1,
s_{\hat{1}}(x_i)=x_i \text{ if } i > 2 \;
({\rm type}\; D).
\end{array}
$$
Based on these choice of the action
of the simple 
transpositions, the divided difference operators
are defined uniquely. 
As was remarked in \cite{FK2},
it is easy to see  that for root systems of 
types $B,C$ (and $D$) it is impossible to find  ``good''  representatives for the 
Schubert classes which satisfy the properties $(0),(1),(2),(3)$ listed above. 
Nevertheless in \cite{FK2} the authors introduce the so called Schubert polynomials of 
the first kind with nice combinatorial properties including those 
$(0),(2),(3),(4_w)$, and therefore suitable for computation of the triple intersection 
numbers for Schubert varieties of classical type, the main Problem of the Schubert Calculus, 
see \cite{FK2} for details.

In \cite{BH} the authors defined
 certain action of Weyl group on 
the ring of supersymmetric functions of infinite number
of variables
$\Gamma=(\Z[x_1,x_2,\ldots])^{SS}$ and define another family of 
Schubert polynomials, where $SS$ means supersymmetric 
(for detail see \S 4).

In \cite{IMN} Ikeda, Mihalcea and the second author
defined and studied the double
Schubert polynomials of type $B,C,D$ using 
localization map of equivariant cohomology.
For $K$-theory there is analogous map and
the image has the
so called
Goresky-Kottwitz-MacPherson property \cite{GKM}.
As mentioned for the case of Grassmannians in \cite{IN},
the Schubert classes can be characterized by 
recurrence relations.
These are essentially done already by  Kostant-Kumar \cite{KK}
and used in \cite{LSS},
see \S 6 for more details.

In conclusion
in the present paper 
we used an algebro-combinatorial construction of \cite{FK1}
to extend the algebro-geometric \cite{IN} constructions of 
the double Schubert polynomials 
of  types $B,C,D,$
to get double Grothendieck polynomials
which represent the Schubert classes in the 
K-theory rings of the types $B,C$ and $D$ full flag varieties. 
Some of these polynomials also appear in more geometric context
of connective $K$-theory 
of (non-maximal) Grassmannians in \cite{HIMN}
, where the parameter $\beta=a_{1,1}$ has its meaning.

The formulas obtained $(3)$ and $(6)$ lead to combinatorial descriptions of polynomials in questions in terms of either EYD, or compatible sequences,  
or set-valued tableaux \cite{Bu}.
We expect that after a
certain change of idCoxeter algebra and
replacing $ A \oplus B$ in our formulas $(3)$ in Lemma 9 and $(6)$ in Lemma 10 
by $F(A,B)$,
 where $F(x,y)$
stands for  the universal formal group law, we come to 
formal power series which have a suitable interpretations in the theory of 
algebraic cobordism \cite{LM} 
of flag varieties.        

\subsection{Organization}
In Section 2
we summarize the notations and definitions needed.
In Section 3
we describe some basic properties of IdCoxeter algebra ${\rm Id}_\beta(W)$.
In Section 4
we define $\beta$-supersymmetric functions and $K$-theoretic Schur $P$- $Q$-functions and
$K$-theoretic Stanley symmetric functions.
In Section 5
we introduce the double Grothendieck polynomials
of classical types and some fundamental properties.
In Section 6
we give a geometric interpretation of the double Grothendieck polynomials
using (algebraic) localization map.
In Section 7
we introduce adjoint polynomials which are dual to the double Grothendieck polynomials.
Finally in Section 8
we give two types
 of combinatorial formula for double Grothendieck polynomials
 using compatible sequences and excited Young diagrams.

\section{Definitions and Notations}

In this paper $W=W(X)$ is a Weyl group of type $X=A,B,C,D$.
$I=I^X$ is the set of simple reflections in $W(X)$.
We index the simple reflections
by the same notation as in \cite{IMN} \S3.2.
In particular, for type $B$ and $C$, $s_0$ corresponds to the left most node of the Dynkin diagram
with the relations
$(s_0 s_1)^4=1$ and $(s_0 s_i)^2=1$ for $i\geq 2$.
For type $D$, $s_{\hat{1}}:=s_0 s_1 s_0$ and we consider $W(D)$ as the subgroup of $W(B)$
generated by $s_{\hat{1}},s_1,\ldots$.
For $X=B$ and $C$, the Weyl group 
$W(X_n)=\langle s_0,s_1,\ldots,s_{n-1}\rangle$ is 
the hyperoctahedral group and the elements are 
 realized as signed permutations. (cf. \cite{IMN} \S3.3.)
 (Maximal)
Grassmannian elements of type $B_n$ and $C_n$ are minimal length coset representatives
of
$W(B_n)/S_n=W(C_n)/S_n$
where $S_n=\langle s_1,\ldots,s_{n-1}\rangle$ is the
parabolic subgroup corresponding to the
index $0$.
For a Grassmannian element $w=[\bar{i}_{1},\ldots,\bar{i}_\ell,i_{\ell+1},\ldots, i_{n}]$ of type $X=B,C$, 
where $1\leq i_1,\ldots,i_n\leq n$ are distinct integers with $i_1>\cdots>i_{\ell}$ and $i_{\ell+1}<\cdots<i_n$,
we associate strict partition
{$\lambda_X(w)=(i_1,\ldots,i_{\ell})$}.
(Maximal)
Grassmann elements of type $D_n$ are minimal length coset representatives
of
$W(D_n)/S_n$
where $S_n=\langle s_1,\ldots,s_{n-1}\rangle$ is the
parabolic subgroup corresponding to the
index $\hat{1}$.
For a Grassmannian element $w=[\bar{i}_{1},\ldots,\bar{i}_\ell,i_{\ell+1},\ldots, i_{n}]$ of type $D$, 
where $1\leq i_1,\ldots,i_n\leq n$ are distinct integers with 
$i_1>\cdots>i_{\ell}$ and $i_{\ell+1}<\cdots<i_n$,
we associate strict partition
{$\lambda_D(w)=(i_1-1,\ldots,i_{\ell}-1)$}. Note that for type $D$ case $\ell$ is always even and
we can omit $i_{\ell}-1=0$ when $i_{\ell}=1$.

We use Bruhat order $w\leq v$ on $W(X)$. Then it is known that
for (maximal) Grassmanian elements $w,v\in W(X)$, we have 
$$w\leq v \iff \lambda_X(w)\subset \lambda_X(v).$$
The set of root $\Delta_{X}$ is the set of orbits of simple roots.

Following \cite{FK1}, we prepare some notations.
Let $\beta$ be an indeterminate. We define operations $\oplus$ and $\ominus$ 
as follows.
$$x\oplus y:=x+y+\beta x y, x\ominus y:=(x-y)/(1+\beta y) .$$
We also use the convention that 
$$\bar{x}:=\ominus x=-\frac{x}{1+\beta x}.$$
Then we have $x\op \bar{x}=0$.
For a Weyl group $W$ with the set $I$ of Coxeter generators, 
we define idCoxeter algebra as follows.

\begin{defn}(IdCoxeter algebra)
IdCoxeter algebra ${\rm Id}_\beta(W)$ for $W$ is a $\Z[\beta]$ algebra 
with generators $u_i$
 for each $s_i\in I$
and relations as follows.
\begin{center}
$u_i^2=\beta u_i$,
\\[0.2cm]
$\underbrace{\stackrel{}{u_i u_j u_i \cdots}}_{m_{i,j} \;
\rm terms} 
=\underbrace{\stackrel{}{u_j u_i u_j \cdots}}_{m_{i,j} 
\;\rm terms}$
\;\;if\;$m_{i,j}$ is the order of $s_i s_j$.
\end{center}
\end{defn}
By the braid relation we can define $u_w=u_{{i_1}}\cdots u_{{i_\ell}}$
where $w=s_{i_1}\cdots s_{i_\ell}$ is a reduced expression of $w\in W$.
Then $\{u_{w}\}_{w\in W}$ form a $\Z[\beta]$ basis of ${\rm Id}_\beta(W)$.

For each $s_i\in I^X$, we define divided-difference operator 
$\pi^{(a)}_i$ and $\psi^{(a)}_{i}$
with respect to the variables $a=(a_{1},a_{2},...)$
as follows. Assume that $R\supset \Z[\beta]$ is a ring with a group action of $W(X)$.
We define the action of $W(X)$ on 
$R[a,\bar{a}]:=R[a_1,a_2,...,\bar{a}_1,\bar{a}_2,...]$ as follows.
\begin{defn} The action of $s^{(a)}_i\in I^X$ on the variables
$a_1,a_2,\ldots ,\bar{a}_1,\bar{a}_2,\ldots$.
\begin{itemize}
\item
If $i\geq 1$,
 $s^{(a)}_i(a_i)=a_{i+1}, s^{(a)}_i(a_{i+1})=a_i$,
$s^{(a)}_i(\bar{a}_i)=\bar{a}_{i+1}, s^{(a)}_i(\bar{a}_{i+1})=\bar{a}_i$,
{ and }\\
 $s^{(a)}_i(a_k)=a_k, s_i(\bar{a}_k)=\bar{a}_k$ for $k\neq i,i+1$.

\item $s^{(a)}_0(a_1)=\bar{a}_1$, $s^{(a)}_0(\bar{a}_1)={a}_1$,
 and $s^{(a)}_0(a_k)=a_k$, 
$s^{(a)}_0(\bar{a}_k)=\bar{a}_k$ for $k>1$.

\item $s^{(a)}_{\hat{1}}(a_1)=\bar{a}_2, s^{(a)}_{\hat{1}}(a_2)=\bar{a}_1,
s^{(a)}_{\hat{1}}(\bar{a}_1)={a}_2, s^{(a)}_{\hat{1}}(\bar{a}_2)={a}_1,
$ and 
$s^{(a)}_{\hat{1}}(a_k)=a_k,
s^{(a)}_{\hat{1}}(\bar{a}_k)=\bar{a}_k 
$ for $k>2$.
\end{itemize}
\end{defn}

We write the induced action on $R[a,\bar{a}]$ by $s^{(a)}_i$.
Divided difference operators $\pi^{(a)}_i$ and $\psi^{(a)}_{i}$
are defined as follows.
For $f\in R[a,\bar{a}]=R[a_1,a_2,...,\bar{a}_1,\bar{a}_2,...]$,
$$\pi^{(a)}_i(f):=\displaystyle
\frac{f-(1+\beta \alpha_i(a)) s^{(a)}_i(f)}{\alpha_i(a)}
\text{ and }
\psi^{(a)}_{i}:=\pi^{(a)}_i+\beta,
$$

\noindent
where $\alpha_i(a)$ is the element in $\Z[\beta][a,\bar{a}]$
 corresponding to the root $\alpha_{i}$,
 i.e.
$\alpha_i(a)=a_i\oplus \bar{a}_{i+1}$ for $i=1,2,...$,
$\alpha^B_0(a)=\bar{a}_1$,
$\alpha^C_0(a)=\bar{a}_1\oplus \bar{a}_1$ and
$\alpha_{\hat{1}}(a)=\bar{a}_1\oplus \bar{a}_2$.
 
 ( Formally we can think as 
 $\alpha_{i}(a)=\displaystyle
 \frac{e^{\beta \alpha_{i}}-1}{\beta}$. \;cf. \cite{FK1}.)

\begin{prop}
We have the following relations of operators:\\
(we wrtite $\pi=\pi^{(a)}$, $\psi=\psi^{(a)}$ for short.)\\
$$\pi^{2}_{i}=-\beta\pi_{i},\; \psi^{2}_{i}=\beta\psi_{i}\text{ for all } s_i\in I^X,$$
$$\underbrace{\stackrel{}{\pi_i \pi_j \pi_i \cdots}}_{m_{i,j} 
\;\rm terms} 
=\underbrace{\stackrel{}{\pi_j \pi_i \pi_j \cdots}}_{m_{i,j} 
\;\rm terms}
, \;
\underbrace{\stackrel{}{\psi_i \psi_j \psi_i \cdots}}_{m_{i,j} 
\;\rm terms} 
=\underbrace{\stackrel{}{\psi_j \psi_i \psi_j \cdots}}_{m_{i,j} 
\;\rm terms}$$
if\;\;$m_{i,j}$ is the order of $s_i s_j$ .

\end{prop}
\begin{proof}
We can check the relations by direct calculations.
\end{proof}

The explicit form of 
$\psi^{(a)}_i$ is as follows,
\begin{center}
$\psi^{(a)}_{i}(f)=\frac{s^{(a)}_i f-f}{a_{i+1}\ominus a_i}$
 for $i\geq 1$
,\\
\indent
$\psi^{(a)}_{0,B}(f)=\frac{s^{(a)}_0 f-f}{a_1}$
,
$\psi^{(a)}_{0,C}(f)=\frac{s^{(a)}_0 f-f}{a_1\oplus a_1}$
and
$\psi^{{(a)}}_{\hat{1}}(f)=
\frac{s^{(a)}_{\hat{1}} f-f}{a_1\oplus a_2}$.
\end{center}

Similarly we can define divided difference operators  $\pi^{(b)}_i$
and 
$\psi^{(b)}_i$
corresponding to  the variables $b_1,b_2,...$ using $s^{(b)}_i$ and $\alpha_i(b)$.

\section{Basic Properties}

We always assume that all the variables $x,y$ or $a,b$ commute
with $u_i$ and consider in a suitable extension
 of the ring of coefficients in ${\rm Id}_\beta(W)$.
Let $h_i(x):=1+xu_i$.
Then it follows that
$h_i(x)h_i(y)=h_i(x\oplus y)$ and $h_i(x)$ is invertible with 
$h_i(x)^{-1}=h_i(\bar{x})$.

\begin{lem}(Yang-Baxter relations \cite{FK3}) 
The following equalities hold.

\noindent
$\begin{array}{cccc}
h_i(x) h_j(y)&=&h_j(y) h_i(x)&m_{i,j}=2\\[0.2cm]
h_i(x) h_j(x\oplus y) h_i(y)
&=&h_j(y) h_i(x\oplus y) h_j(x)&m_{i,j}=3\\[0.2cm]
h_i(x) h_j(x\oplus y) h_i(x\oplus y\oplus y)h_j(y)
&=&h_j(y) h_i(x\oplus y\oplus y) h_j(x\oplus y) h_i(x)&m_{i,j}=4\\
\end{array}
$

\end{lem}

These can be proved by direct calculations. (We omit the case of $m_{i,j}=6$ 
which we don't need.)

\begin{defn} We define the following elements in ${\rm Id}_\beta (W)[x]$ for $W=W(X)$ with $X=A,B,C,D$.
\begin{center}
\begin{minipage}{11cm}
\indent
$A_i^{(n)}(x):=\displaystyle\prod_{k=n-1}^{i} h_k(x)=
h_{n-1}(x) h_{n-2}(x)\cdots h_i(x)$. \;$(i=1,2,...,n-1)$,

${F}^{B}_n(x):=
A_1^{(n)}(x) \;{h_{0}(x)}\; A_1^{(n)}(\bar{x})^{-1}\\
\hspace{1.1cm}=h_{n-1}(x)h_{n-2}(x)\cdots h_1(x) {h_{0}(x)} h_1(x)
\cdots h_{n-2}(x) h_{n-1}(x)
$,

${F}^{C}_n(x):=
A_1^{(n)}(x) \;{h_{0}(x)^2}\; A_1^{(n)}(\bar{x})^{-1}\\
\hspace{1.1cm}
=
h_{n-1}(x)h_{n-2}(x)\cdots h_1(x) {h_{0}(x)^{2}} h_1(x)
\cdots h_{n-2}(x) h_{n-1}(x)
$,

${F}^{D}_n(x):=
A_2^{(n)}(x)\;{h_{\hat{1}}(x)h_1(x)}\; A_2^{(n)}(\bar{x})^{-1}\\
\hspace{1.1cm}=
h_{n-1}(x)\cdots h_2(x){h_1(x) h_{\hat{1}}(x)}
h_2(x)\cdots h_{n-1}(x)
$.
\end{minipage}
\end{center}

\end{defn}
For $1\leq i\leq j$, we abbrebiate 
$$[i,j]_{x}:=h_i(x)h_{i+1}(x)\cdots h_{j}(x)\text{ and }
[j,i]_{x}:=h_j(x)h_{j-1}(x)\cdots h_{i}(x).$$

\begin{lem}
For $1\leq i\leq j$, we have
$[i,j]_{x} [j,i]_{y}=
 [j,i]_{y} [i,j]_{x}$.
\end{lem}
\begin{proof}
We will prove by induction on $j-i$.
When $j-i=0$, i.e. $i=j$, it is trivial.
When $j-i=1,$
$[i,i+1]_x [i+1,i]_y=[i+1,i]_y [i,i+1]_x$ by Yang-Baxter relation.
For $j-i=k\leq 2$,
we can use
 induction hypothesis
and  Yang-Baxter relation again to get 

\noindent
$[i,j]_{x} [j,i]_{y}=[i]_{x}[i+1,j]_{x} [j,i+1]_{y} [i]_{y}
=[i]_{x} [j,i+2]_{y}[i+1]_y [i+1]_x  [i+2,j]_{x}  [i]_{y}
=[j,i+2]_{y} [i]_{x}  [i+1]_x [i+1]_y [i]_y [i+2,j]_{x}
=[j,i+2]_y [i+1,i]_y [i,i+1]_x [i+2,j]_x=
[j,i]_y [i,j]_x$.

\end{proof}

\begin{lem} We have the following equalities.
\indent

(1) 
$A_i^{(n)}(x) A_i^{(n)}(y)=A_i^{(n)}(y) A_i^{(n)}(x)$,

(2) $F^{X}_n(x)F^{X}_n(y)=F^{X}_n(y)F^{X}_n(x)$ for $X=B,C,D$,

(3) $F^{X}_n(x)F^{X}_n(\bar{x})=1$.

\end{lem}

\begin{proof}
(1) As $A_i^{(n)}(x)=[n-1,i]_{x}$ and
$A_i^{(n)}(y)^{-1}=[i,n-1]_{\bar{y}}$, it follows from Lemma 2.

(2) Using (1) and Yang-Baxter relations again,  we can show the equalities as follows. For $X=B$, we have
\vspace{0.5cm}

$F_{n}^B(x)F_{n}^B(y)$\\
\begin{minipage}{13cm}
$\displaystyle\begin{array}{ccl}
&=&[n-1,1]_x [0]_x [1,n-1]_x [n-1,1]_y [0]_y [1,n-1]_y\\
&=&[n-1,1]_x [0]_x [n-1,1]_y [1,n-1]_x  [0]_y [1,n-1]_y\\
&=&[n-1,1]_x [n-1,2]_y [0]_x [1]_y [1]_x  [0]_y [2,n-1]_x[1,n-1]_y\\
&=&[n-1,1]_x [n-1,1]_y [1]_{\bar{y}} [0]_x [1]_y [1]_x  [0]_y
[1]_{\bar{x}} [1,n-1]_x[1,n-1]_y\\
&=&[n-1,1]_y [n-1,1]_x 
[1]_{\bar{y}} [1]_{\bar{x}}[1]_x[0]_x [1]_x [1]_y  [0]_y[1]_y[1]_{\bar{y}} [1]_{\bar{x}} 
[1,n-1]_y [1,n-1]_x\\
&=&[n-1,1]_y [n-1,1]_x 
[1]_{\bar{y}} [1]_{\bar{x}}[1]_y[0]_y [1]_y [1]_x  [0]_x [1]_x [1]_{\bar{y}} [1]_{\bar{x}} 
[1,n-1]_y [1,n-1]_x\\
&=&[n-1,1]_y [n-1,2]_x 
[0]_y [1]_y [1]_x  [0]_x 
[2,n-1]_y [1,n-1]_x\\
&=&[n-1,1]_y  
[0]_y [n-1,2]_x[1]_x [1]_y [2,n-1]_y [0]_x 
 [1,n-1]_x\\
&=&[n-1,1]_y  
[0]_y [1,n-1]_y [n-1,1]_x  [0]_x [1,n-1]_x\\
&=&F_{n}^B(y)F_{n}^B(x)\\
\end{array}
$
\end{minipage}

\noindent
Similar arguments with appropriate modifications will give $X=C,D$ cases.
The essential equalities to be used are
$$h_1(x\oplus \bar{y}) h_0(x\oplus x) h_1(x\oplus y) h_0(y\oplus y) h_1(\bar{x}\oplus y)
=h_0(y\oplus y) h_1(x\oplus y) h_0(x\oplus x)$$
and
$$h_2(x\oplus \bar{y}) h_1(x) h_{\hat{1}}(x) h_2(x\oplus y) h_1(y) h_{\hat{1}}(y) h_2(\bar{x}\oplus y)
=h_1(y) h_{\hat{1}}(y) h_2(x\oplus y) h_1(x) h_{\hat{1}}(x).$$

(3) This  follows essentially by the relation 
$h_i(x) h_i(\bar{x})=1$.

\end{proof}

\section{$\beta$-supersymmetric functions}

\begin{defn}
$\beta$-supersymmetric function with respect to variables $x_1,x_2,\ldots,x_n$
is a symmetric function $f(x_1,x_2,\ldots,x_n)$ on $x_1,x_2,\ldots,x_n$ which satisfies 
the following cancellation property:
$$f(t,\bar{t},x_3,...,x_n)=f(0,0,x_3,...,x_n) \text{ for every } t.$$
\end{defn}
\begin{rem}
 The $\beta$-supersymmetric property is translated to 
 usual supersymmetricity
 by the change of variables 
$x_i$ to  $\displaystyle\frac{e^{\beta x_i}-1}{\beta}$.
If $\beta=0$, then the $\beta$-supersymmetric property
becomes usual supersymmetric property, i.e.\\
$f(t,-{t},x_3,...,x_n)=f(0,0,x_3,...,x_n) \text{ for every } t.$
\end{rem}

Let
${
SS_{\beta}(x_1,\ldots,x_n)}:=\{f\in \Z[\beta][x_1,...,x_n]\;|\; 
f:\beta\text{-supersymmetric}\}$
and set
$\bSS_{\beta}(x):=
\displaystyle\lim_{\ln} SS_{\beta}(x_1,\ldots,x_n)$.
Then 
$\bSS_{\beta}(x)$ is the ring of $\beta$-supersymmetric functions.
(It is denoted as $G\Gamma$ in \cite{IN}.)
If $\beta=0$  this becomes 
the ring $\Gamma'$ in  \cite{IMN}.

\subsection{$K$-theoretic Schur functions 
$GP_{\lambda}(x),GQ_{\lambda}(x)$.}

In \cite{IN}  $\beta$-supersymmetric functions 
$GP_{\lambda}(x), GQ_{\lambda}(x)$ are defined.
Let $b_1,b_2,...$ be indeterminates, 
and set
$[x|b]^{k}=(x\oplus b_1)\cdots (x\oplus b_{k})$ and
$[[x|b]]^{k}=(x\oplus x)(x\oplus b_1)\cdots (x\oplus b_{k-1})$.

Let $SP_n$ be the set of strict partitions of length
at most $n$.
i.e. $\lambda=(\lambda_1>\lambda_2>\cdots >\lambda_r>0)$ is an integer sequence
such that  $r\leq n$.
We also set
$SP=\bigcup_n SP_n$.

\begin{defn}(Ikeda-Naruse \cite{IN})
For a strict partition $\lambda\in SP_n$,
$$
\begin{array}{lll}
GP_\lambda(x_1,\ldots,x_n|b)&:=&
\displaystyle\frac{1}{(n-r)!}
\sum_{w\in S_n}
w \left(
\prod_{1\leq i\leq r}\left(
[x_i|b]^{\lambda_i}
\prod_{i<j\leq n} \frac{x_i\oplus x_j}{x_i\ominus x_j}
\right)
\right),
\\[0.7cm]
GQ_\lambda(x_1,\ldots,x_n|b)&:=&
\displaystyle\frac{1}{(n-r)!}
\sum_{w\in S_n}
w \left(
\prod_{1\leq i\leq r}\left(
[[x_i|b]]^{\lambda_i}
\prod_{i<j\leq n} \frac{x_i\oplus x_j}{x_i\ominus x_j}
\right)
\right),
\end{array}
$$
\noindent
where $w\in S_n$ acts $x_1,\ldots,x_n$ as permutation of indices.

\end{defn}

We also put
$$
\begin{array}{cll}
GP_\lambda(x_1,\ldots,x_n)&:=&GP_\lambda(x_1,\ldots,x_n|0),\\[0.1cm]
GQ_\lambda(x_1,\ldots,x_n)&:=&GQ_\lambda(x_1,\ldots,x_n|0),\\[0.1cm]
GP_\lambda(x)&:=&\displaystyle\lim_{\ln} GP_\lambda(x_1,...,x_n),
\\
GQ_\lambda(x)&:=&\displaystyle\lim_{\ln} GQ_\lambda(x_1,...,x_n),\\
GP_\lambda(x|b)&:=&\displaystyle\lim_{\ln} GP_\lambda(x_1,...,x_{2n}| b),
\text{ and }\\
GQ_\lambda(x|b)&:=&\displaystyle\lim_{\ln} GQ_\lambda(x_1,...,x_n|b).
\end{array}
$$
N.B. $GP_\lambda(x_1,\ldots, x_n |b)$ has only mod 2 stability (cf. \cite{IN} Remark 3.1), and  
we define $GP_\lambda(x|b)$ to be the even limit as in \cite{IN}.

\begin{exam}
The followings are some examples of $GP,GQ$.
\begin{center}
\begin{minipage}{9cm}
$GP_{1}(x_1,\ldots,x_n)=x_1\oplus x_2\oplus \cdots \oplus x_n$.

$GQ_{1}(x_1,\ldots,x_n)=(x_1\oplus x_1)\oplus (x_2\oplus x_2) \oplus \cdots 
\oplus (x_n\oplus x_n)$.
\end{minipage}
\end{center}

\end{exam}

\begin{lem} (\cite{IN} Theorem 3.1, Proposition 3.2.)The followings hold.

(1) $GP_\lambda(x_1,\ldots,x_n)$ and $GQ_\lambda(x_1,\ldots,x_n)$ are 
$\beta$-supersymmetric functions.

(2) $\{GP_\lambda(x_1,\ldots,x_n)\}_{\lambda\in SP_n}$ 
forms a basis of $SS_\beta(x_1,\ldots,x_n)$ over $\Z[\beta]$.

(3) Let $SS^{C}_\beta(x_1,\ldots,x_n)$ be 
the $\Z[\beta]$-subspace of 
$SS_\beta(x_1,\ldots,x_n)$ spanned
by $GQ_\lambda(x_1,\ldots,x_n) {(\lambda\in SP_n)}$.
Then
$\{GQ_\lambda(x_1,\ldots,x_n)\}_{\lambda\in SP_n}$ 
forms a basis of $SS^{C}_\beta(x_1,\ldots,x_n)$ over $\Z[\beta]$.

\end{lem}

\begin{rem}
We remark that the definition of $\beta$-supersymmetry and 
the polynomials $GP_\lambda,GQ_\lambda$
can be generalized in more general setting such as algebraic cobordism \cite{LM},
cf. \cite{NN}.
\end{rem}

According to \cite{IN} (6.5),
we define an action of $W(X)$ 
as follows.

\begin{defn}
The action of Weyl group $W(X_n)$ on 
${SS_{\beta}(x_1,\ldots,x_n)}\otimes_{\Z[\beta]} \Z[\beta][a,\bar{a}]
\otimes_{\Z[\beta]} 
\Z[\beta][b,\bar{b}]
$ is derived from the action as follows (together with Definition 2).
For $f(x)\in \bSS_\beta(x)$,
$$s^{(a)}_i f(x)=f(x)=s^{(b)}_i f(x) \text{ for } i\geq 1,$$\\[-1.2cm]
$$s^{(a)}_0 f(x)=f(a_1,x),\;
s^{(b)}_0 f(x)=f(b_1,x), $$\\[-1.2cm]
$$s^{(a)}_{\hat{1}} f(x)=f(a_1,a_2,x),\;
s^{(b)}_{\hat{1}} f(x)=f(b_1,b_2,x).$$
\end{defn}
\noindent
These actions  can be clarified 
by the change of variables explained in the second definition below 
(cf.\S5.2 Remark 4).

\begin{lem}(\cite{IN} Theorem 6.1 and Theorem 7.1)
$GP_\lambda(x|b)$ and $GQ_\lambda(x|b)$ are characterized by (left) divided difference
relations and initial conditions.
i.e.
For a maximal Grassmannian element $w\in W(X)/S_\infty$ and $s_i\in I^{X}$,
$$\pi_i^{(b)} GX_{\lambda(w)}(x|b)=\begin{cases}
GX_{\lambda(s_i w)}(x|b)&\mbox{if}\quad s_i w<w \\
-\beta \;GX_{\lambda(w)}(x|b) & \mbox{if}\quad s_i w\geq w
\end{cases}$$
and
$$ GX_\emptyset(x|b)=1,$$
\noindent
where
$GB_\lambda(x|b)=GP_\lambda(x|0,b),
GC_\lambda(x|b)=GQ_\lambda(x|b),
GD_\lambda(x|b)=GP_\lambda(x|b)
$.
\end{lem}

\subsection{$K$-theoretic Stanley symmetric functions ${\cal F}^X_w(x)$, 
$X=B,C,D$}

\begin{defn} For $X=B,C,D$, we define
$${ \bF}^{X}_n(x):=
F^{X}_n(x_1,\ldots,x_n)=
\displaystyle
\prod_{i=1}^{n} F^{X}_n(x_i)
\;
\text{ and }\;
{ \bF}^{X}_\infty(x)
:=\displaystyle\lim_{\ln} { \bF}^{X}_n(x).
$$
\noindent
Using these
we define ${\cal F}^X_{w}(x_1,...,x_n)$ and
${\cal F}^X_{w}(x)$
 by the following expansions.
$${ \bF}^{X}_n(x)
=\displaystyle
\sum_{w\in W(X_n)} {\cal F}^X_{w}(x_1,...,x_n) u_w
\;,\:
{ \bF}^{X}_\infty(x)
=\displaystyle
\sum_{w\in W(X)} {\cal F}^X_{w}(x) u_w.
$$

\end{defn}
\noindent
By definition ${\cal F}^X_{w}(x_1,...,x_n)$ are weakly stable but
 not strongly stable (cf. Definition 7).
$\F^X_w(x)$ is a $K$-theoretic analogue of Stanley symmetric function
of type $X=B,C,D$. (cf. \cite{BH})

\begin{lem} For each $w\in W(X_n)$,
${\cal F}^{X}_{w}(x_1,x_2,\dots,x_n)$ is a 
$\beta$-supersymmetric function.
\end{lem}

\begin{proof}
  This follows from Lemma 3 (2) and (3).
\end{proof}

\begin{lem}

\begin{itemize}
\item[(0)] For $X=B,C,D$, we have
$$\F^{X}_{w^{-1}}(x_1,x_2,\dots,x_n)=\F^{X}_w(x_1,x_2,\dots,x_n).$$

\item[(1)]
For $X=B,C,D$, 
$\F^{X}_w(x_1,x_2,\dots,x_n)$ can be expanded in 
$GP_{\lambda}(x_1,x_2,\dots,x_n)$ with coefficients in $\Z[\beta]$.

\item[(2)] For a (maximal) Grassmannian element $w\in W(X_n)$,
\begin{itemize}
\item[] $\F^{B}_{w}(x_1,x_2,\dots,x_n)=GP_{\lambda_B(w)}(x_1,x_2,\dots,x_n),$
 
\item[] $\F^{C}_{w}(x_1,x_2,\dots,x_n)=GQ_{\lambda_C(w)}(x_1,x_2,\dots,x_n),$

\item[] $\F^{D}_{w}(x_1,x_2,\dots,x_n)=GP_{\lambda_D(w)}(x_1,x_2,\dots,x_n).$
\end{itemize}

\end{itemize}

\end{lem}

\begin{proof}

(0) This follows from the symmetry of $\bF_{n}^X(x)$, i.e.
if a coefficient of $u_w$ comes from the product 
$u_{i_1} u_{i_2}\cdots u_{i_k}$ in 
${ F}^{X}_n(x_1)\cdots{ F}^{X}_n(x_n)$,
then the same coefficient appears in the product
$u_{i_{k}} u_{i_{k-1}}\cdots u_{i_1}$ in 
${ F}^{X}_n(x_n)\cdots{ F}^{X}_n(x_1)$,
by picking up the symmetric positions. 

(1) This follows from Lemma 4  (2), because $\bF^{X}_n(x)$ is 
$\beta$-supersymmetric.

(2) This  is Proposition 5 below with $b=0$.

\end{proof}
\begin{rem}
We state conjecture that the coefficients in the expansion of 
(1) are positive, i.e. the coefficients will be polynomials in $\beta$
 with nonnegative integers.
This will be a consequence of $K$-theory analogue of ``transition equation" for type
$B,C,D$.(cf. \cite{IMN})
\end{rem}

\begin{exam} Belows are some examples of $\F^{X}_{w}(x_1,\ldots,x_n)$.

$\F^{B}_{s_0}(x_1,\ldots,x_n)=GP_1(x_1,\ldots,x_n)$.

$\F^{C}_{s_0}(x_1,\ldots,x_n)=GQ_1(x_1,\ldots,x_n)=
2GP_1(x_1,\ldots,x_n)+\beta GP_2(x_1,\ldots,x_n)$.

$\F^{D}_{s_{\hat{1}}}(x_1,\ldots,x_n)=GP_1(x_1,\ldots,x_n)$.
\end{exam}

\section{Main results}

In this section we define the main object of this paper,
the double Grothendieck polynomials of classical types
$\G^{X}_w(a,b;x)$, $(X=B,C,D)$
,
and  show some of their
fundamental
properties.

First we recall the type $A$ Grothendieck polynomials
$\G_{w}^{A}(a)$
cf. \cite{FK2}.
These polynomials satisfy the strong stability in the
following sense.

\begin{defn}
Fix an element $w\in W(X)$ ($X=A,B,C,D$).
Suppose that
for each $n$ such that $w\in W(X_n)$
we have given a polynomial  $f_w^{(n)}\in R[x_1,x_2,\ldots,x_n]$.
Then

(1) $\left\{f_w^{(n)}\right\}_{n\geq 1}$ is called weakly stable (with respect to $x$) if for all $m>n$ we have
$f_w^{(m)} |_{x_{n+1}=\cdots =x_{m}=0}=f_w^{(n)}$.

(2) $\left\{f_w^{(n)}\right\}_{n\geq 1}$ is called strongly stable (with respect to $x$) if
for all $m>n$ we have
$f_w^{(m)}=f_w^{(n)}$.
\end{defn}

We set
${G}_{A_{n-1}}(a_1,...,a_{n-1}):=
A_{1}^{(n)}(a_1) A_{2}^{(n)}(a_2)\cdots A_{n-1}^{(n)}(a_{n-1})$.
Then for $w\in S_n$, we define $\G^{A_{n-1}}_w(a)$ 
by the following equation.
$${G}_{A_{n-1}}(a_1,...,a_{n-1})=\displaystyle
\sum_{w\in S_n}\G^{A_{n-1}}_w(a) u_w.$$
Furthermore,
we can consider 
$G_A(a):= \displaystyle \lim_{\ln} G_{A_{n-1}}(a_1,...,a_{n-1})$ and get 
$\G_{w}^{A}(a)$ by
$${G}_{A}(a)=\displaystyle
\sum_{w\in S_\infty}\G^{A}_w(a) u_w.$$
\noindent
It is easy to see that  for $w\in S_n$ and $m>n$,
we have $\G^{A_{m-1}}_w(a)=\G^{A_{n-1}}_w(a)=\G^{A}_w(a)$,
therefore
the type $A$ Grothendieck polynomials are strongly stable.
Recall that
the type $A_{n-1}$ double Grothendieck polynomials 
$\G^{A_{n-1}}_w(a,b)$
are defined as follows.
$$
{G}_{A_{n-1}}(\bar{b}_1,...,\bar {b}_{n-1})^{-1}
{G}_{A_{n-1}}(a_1,...,a_{n-1})
=\sum_{w\in S_{n}}
\G^{A_{n-1}}_w(a,b) u_w
=:{G}^{A}_{n-1}(a,b)
.
$$

\begin{lem}We have the following equation.

\noindent
$\begin{array}{lll}
&&{G}_{A_{n-1}}(\bar{b}_1,...,\bar {b}_{n-1})^{-1}
{G}_{A_{n-1}}(a_1,...,a_{n-1})\\
&=&
h_{n-1}(a_1\oplus b_{n-1})\\
&&h_{n-2}(a_1\oplus b_{n-2}) h_{n-1}(a_2\oplus b_{n-2})\\
&&\hspace{1.3cm}\vdots\\
&&h_{2}(a_1\oplus b_{2})h_{3}(a_2\oplus b_{2})\cdots  h_{n-1}(a_{n-2}\oplus b_{2})\\
&&h_{1}(a_1\oplus b_{1})h_{2}(a_2\oplus b_{1})\cdots  h_{n-2}(a_{n-2}\oplus b_{1}) h_{n-1}(a_{n-1}\oplus b_{1})\\[0,5cm]
&=&
h_{n-1}(a_1\oplus b_{n-1})h_{n-2}(a_1\oplus b_{n-2})\cdots  h_{2}(a_{1}\oplus b_{2}) 
h_{1}(a_{1}\oplus b_{1})\\
&&h_{n-1}(a_2\oplus b_{n-1})h_{n-2}(a_2\oplus b_{n-2})\cdots  h_{2}(a_{2}\oplus b_{1})\\
&&\hspace{1.3cm}\vdots\\
&&h_{n-1}(a_{n-2}\oplus b_{2}) h_{n-2}(a_{n-2}\oplus b_{1})\\
&&h_{n-1}(a_{n-1}\oplus b_{1})\\
\end{array}
$
\end{lem}

\begin{proof}
We can use Yang-Baxter relations many times to transform the given product.

\end{proof}


\subsection{The first definition}
\begin{defn} We define for $X=B,C$ or $D$,
$${ G}_{n}^X(a,b;x):=
{G}_{A_{n-1}}(\bar{b}_1,...,\bar {b}_{n-1})^{-1}
{\bF}_n^X(x)
{G}_{A_{n-1}}(a_1,...,a_{n-1})$$
and  define $\G^{X}_{n,w}(a,b;x)$ as
the coefficient of $u_w$.

$${ G}^X_{n}(a,b;x)=
\displaystyle \sum_{w\in W(X_n)} \G^{X}_{n,w}(a,b;x) u_w.$$
\noindent
Furthermore, we define $\G^{X}_w(a,b;x)$ by
$$
G_A(\bar{b})^{-1}
\bF^{X}_\infty(x)
G_A(a)=
\displaystyle \sum_{w\in W(X)} \G^{X}_w(a,b;x) u_w.$$
\end{defn}
By the definition we can see
$ \G^{X}_{n,w}(a,b;x) \in 
{SS_{\beta}(x_1,\ldots,x_n)}
[a_1,\dots,a_{n-1},b_1,\dots,b_{n-1}]$
and
$ \G^{X}_{w}(a,b;x) \in 
{\bSS_{\beta}(x)}
[a_1,\dots,a_{n-1},b_1,\dots,b_{n-1}]$
for $w\in W(X_n)$, i.e.
a polynomial in $a_1,\dots,a_{n-1},b_1,\dots,b_{n-1}$ with coeffcients
in ${SS_{\beta}(x_1,\ldots,x_n)}$ or ${\bSS_{\beta}(x)}$.
When we set $\beta=0$, $a_i=z_i$ and $b_i=-t_i$,
$\G^{X}_{w}(a,b;x)$ becomes
the double Schubert polynomial of classical type
defined in \cite{IMN}.
The main features of these polynomials are summarized in the following.
\begin{thm} For $X=B,C,D$,
$\G^{X}_{w}(a,b;x)$ satisfies the 
$K$-theoretic (double) version of
the properties
$(0),(1),(2),(3),(4_s)$ listed in Introduction.

\end{thm}

\begin{proof}

(0) follows by the definition. (We set $\deg a_i=\deg b_i=\deg x_i=1, \deg \beta=-1$.)

(1) $K$-theoretic divided difference compatibility follows by Corollary 1 below. 

(2) follows by Theorem 2 in the next section.

(3) follows by the definition. (Here nonnegativity means that in
$\G^{X}_{w}(a,b;x)$ each coefficient of monomials in variables $a,b,x$  is a polynomial in $\beta$ with nonnegative integer coefficients.)
For explicit combinatorial formulas see Theorem 5 and 6 in section 8.

($4_s$) follows by Proposition 4 below.

\end{proof}

We will write $w\star v=z$ (called Demazure product) if
$u_{w}u_{v}=\beta^{\ell(w)+\ell(v)-\ell(z)} u_{z}$.
It is associative and $w\star v=w v$ when $\ell(w)+\ell(v)=\ell(w v)$.

\begin{prop} For $X=B,C,D$ and $w\in W(X)$, we have
$${\cal G}^X_{w}(a,b;x)=\displaystyle
\sum_{(v_1,u,v_2)\in R(w)} 
{\cal G}^{A}_{v_1^{-1}}(b) {\cal F}^{X}_{u}(x){\cal G}^{A}_{v_2}(a),
$$
where $R(w)=\{
(v_1,u,v_2)\in S_\infty\times W(X)\times S_\infty\;|\;
v_1\star u\star v_2=w
\}$.

\end{prop}

\begin{prop}
We have 

$\pi^{(a)}_i { G}^X_n(a,b;x)={G}^X_n(a,b;x)(u_i-\beta)$
and 
$\pi^{(b)}_i {G}^X_n(a,b;x)=(u_i-\beta){G}^X_n(a,b;x)$.

\end{prop}

\begin{proof}
We will prove $\psi^{(a)}_i { G}^X_n(a,b;x)={G}^X_n(a,b;x)u_i$. 
Recall the explicit formula of $\psi_i$ after the Proposition 1.
${G}_{A_{n-1}}(\bar{b})^{-1}$ is invariant for the action of $s_i^{(a)}$, $s_i\in I^X$.
For $i>0$,
$\psi^{(a)}_i \bF_n^X(x)=\bF_n^X(x)$ and
$\psi^{(a)}_i {G}_{A_{n-1}}(a)={G}_{A_{n-1}}(a) u_i$ 
(cf. \cite{FK1}),
therefore

\begin{center}
$\begin{array}{lll}
\psi^{(a)}_i \bF_n^X(x){G}_{A_{n-1}}(a)&=&\bF_n^X(x){G}_{A_{n-1}}(a) u_i
.
\\[0.2cm]
\psi_{0,B}^{(a)} (\bF_n^B(x){G}_{A_{n-1}}(a))
&=&
\frac{\bF^B_n(x) F^B_n(a_1){G}_{A_{n-1}}(\bar{a}_1,a_2,...,a_{n-1}) -
\bF_n^B(x) {G}_{A_{n-1}}(a))}{a_1}\\
&=&\bF_n^B(x){G}_{A_{n-1}}(a) u_0.
\\
\psi_{0,C}^{(a)} (\bF_n^C(x){G}_{A_{n-1}}(a))
&=&
\frac{\bF^C_n(x) F^C_n(a_1){G}_{A_{n-1}}(\bar{a}_1,a_2,...,a_{n-1}) -\bF_n^C(x) {G}_{A_{n-1}}(a)}{a_1\oplus a_1}\\
&=&\bF_n^C(x){G}_{A_{n-1}}(a) u_0.
\\
\psi_{\hat{1}}^{(a)} (\bF_n^D(x){G}_{A_{n-1}}(a))
&=&
\frac{\bF^D_n(x)F^D_n(a_1,a_2){G}_{A_{n-1}}(\bar{a}_2,\bar{a}_1,...,a_{n-1}) -\bF_n^D(x){G}_{A_{n-1}}(a)}{a_1\oplus a_2}\\
&=&\bF_n^D(x){G}_{A_{n-1}}(a) u_{\hat{1}}.
\end{array}
$
\end{center}

\noindent
Similar arguments hold for the action of $\psi^{(b)}_i$.
\end{proof}

\begin{cor}
$$\pi_{i}^{(a)}\G_{w}^{X}(a,b;x)= \begin{cases}
{\G}_{ws_{i}}^{X}(a,b;x)& \text{if $\ell(ws_{i})=\ell(w)-1$}, \\
-\beta \G_{w}^{X}(a,b;x) & \text{otherwise} 
\end{cases}
$$
and
$$\pi_{i}^{(b)}\G_{w}^{X}(a,b;x)= \begin{cases}
{\G}_{s_{i}w}^{X}(a,b;x)& 
\text{if $\ell(s_{i}w)=\ell(w)-1$}, \\
-\beta \G_{w}^{X}(a,b;x) & \text{otherwise} 
\end{cases}.
$$
\end{cor}

\begin{prop}(strong stability)

$\G^X_w(a,b;x)$ has strong stability
with respect to $a$ and $b$ (cf. Definition 7),
i.e.
if $i_{n+1}:W(X_n)\to W(X_{n+1})$ is the natural inclusion,
then
$$\G^X_{i_{n+1}(w)}(a,b;x)
=\G^X_w(a,b;x)\in \bSS_\beta(x)[a_1,b_1,a_2,b_2,\ldots].$$
\noindent
This means that $\G^X_w(a,b;x)$ does not depend on $n$ for 
$w\in W(X_n)$.
\end{prop}

The special case of $w$ being a Grassmannian permutation,
$\G^X_w(a,b;x)$ is
the $K$-theoretic analogue of
factorial Schur $P$- or $Q$-function in \cite{IN}.

\begin{prop}(Grassmannian elements)
For a Grassmannian element $w\in W(X)$
($X=B,C,D$)
, we have the following
equalities.

\begin{center}
\begin{minipage}{7cm}
\begin{itemize}
\item[]
$\G^B_w(a,b;x)=GP_{\lambda_B(w)}(x|0,b)$,
\item[]
$\G^C_w(a,b;x)=GQ_{\lambda_C(w)}(x|b)$,
\item[]
$\G^D_w(a,b;x)=GP_{\lambda_D(w)}(x|b)$.
\end{itemize}
\end{minipage}
\end{center}

\end{prop}

\begin{proof}
In \cite{IN} Corollary 7.1, the map $\Phi:
G\Gamma^X\to {\rm Fun}(\sSP,\sR)$ is defined and 
indicated that it is injective.
Let $w\in W(X)$ be a Grassmannian element 
with corresponding
strict partiton $\lambda=\lambda_X(w)$.
Then
$\G^X_w(a,b;x)$ is in 
$G\Gamma^X={\bSS_{\beta}(x)}\otimes \Z[\beta][b,\bar{b}]$ 
and
satisfy the left divided difference property (Corollary 1).
This means that 
$\G^X_w(a,b;x)=GX_\lambda(x|b)$
by the Theorem 7.1 of \cite{IN}.
\end{proof}

\subsection{The second definition}

As in \cite{FK2}, we can use ``change of variables" for $x_i$, $i=1,2.\ldots$
to define the double  Grothendieck polynomial ${\cal G}^{X_n}_w(a,b)$
with two sets of variables $a,b$ as follows.
We just write $F_n$ for $F_n^X$.
$$F_n(x_i)=
\sqrt{\rule{0pt}{2ex}F_n(\bar{a}_i)}\sqrt{F_n(\bar{b}_i)}$$
\noindent
where 
$$\sqrt{1+T}=1+\frac{T}{2}-\frac{T^2}{8}+\frac{T^3}{16}
-\frac{5T^4}{128}\cdots
\text{ (Taylor expansion).}$$
We will also write 
$$\sqrt{F_n(a_1)}\sqrt{F_n(a_2)}\ldots\sqrt{F_n(a_n)}\text{ as }
\sqrt{F_n(a_1,a_2,\ldots, a_n)}$$ because $F_n(a_i)$ commuts with each other.
Note that
$\sqrt{F_n(t)}\in\Q[\beta][[t]]\otimes {\rm Id}_\beta(W_n)$.
\begin{rem}
By the definition of the action of $s_0$ and
the cancellability of $F_n$, we have $s^{(a)}_0 (\sqrt{F_n(\bar{a}_1,\bar{a}_2,\ldots,\bar{a}_n)})
=\sqrt{F_n({a}_1,\bar{a}_2,\ldots,\bar{a}_n)}
 =\sqrt{F_n({a}_1,a_1,\bar{a}_1,\bar{a}_2,\ldots,\bar{a}_n)}
=F_n(a_1)\sqrt{F_n(\bar{a}_1,\bar{a}_2,\ldots,\bar{a}_n)}$.
This explains the action $s_0^{(a)}(F_n(x_1,x_2,\ldots,x_n))=F_n(a_1,x_1,x_2,\ldots,x_n)$
and $s_0^{(b)}(F_n(x_1,x_2,\ldots,x_n))=F_n(b_1,x_1,x_2,\ldots,x_n)$.
The action of $s_{\hat{1}}^{(a)}$ and $s_{\hat{1}}^{(b)}$
are the like.
\end{rem}

\begin{defn}
Let $X=B,C,D$. 
For $w\in W_n^X$, we
define  ${ G}^{X}_{n}(a)$ and ${ G}_n^X(a,b)$ as follows.
$${ G}^{X}_{n}(a):=\sqrt{{ F}^{X}_n(\bar{a}_1,...,\bar{a}_n)}
{ G}_{A_{n-1}}(a)\hspace{0.5cm}
\text{ and  }\hspace{0.5cm}
{ G}_n^X(a,b):={ G}_n^X(\bar{b})^{-1} { G}_n^X(a).$$

\noindent
By expanding these in terms of $u_w$,  we can define 
${\cal G}^{X}_{n,w}(a)$ and 
${\cal G}^{X}_{n,w}(a,b)$ by
$${ G}_n^X(a)=\sum_{w\in W(X)} {\cal G}^{X}_{n,w}(a) u_w
\hspace{0.5cm}
\text{ and }\hspace{0.5cm}
{ G}_n^X(a,b)=
\sum_{w\in W(X_n)} {\cal G}^{X}_{n,w}(a,b) u_w.$$
\end{defn}

\begin{rem}
This double Grothendieck polynomial 
${\cal G}^{X}_{n,w}(a,b)$ is essentially the same as
defined in \cite{Ki}.
This has weak stability. i.e. $\G^{X}_{n,w}=\G^{X}_{n+1,w}|_{a_{n+1}=b_{n+1}=0}$ for $w \in W(X_n)$. 
But it doesn't have strong stability.
\end{rem}

Note that for $w\in W(X_n)$,
then 
\begin{itemize}
\item[] ${\cal G}^{X}_{n,w}(a)\in \Q[\beta][[a_1,...,a_n,
\bar{a}_1,...,\bar{a}_n]]$
{ and }
\item[] 
$
{\cal G}^{X}_{n,w}(a,b)\in \Q[\beta][[a_1,...,a_n,
\bar{a}_1,...,\bar{a}_n,
b_1,...,b_n,
\bar{b}_1,...,\bar{b}_n]].$
\end{itemize}

\vspace{0.5cm}

\begin{exam}
The followings are some examples of 
${\cal G}^{X}_{n,w}(a,b)$.

$\G^{B}_{2,s_0}(a,b)=
\frac{\sqrt{1+ (\bar{a}_1\oplus \bar{a}_2\oplus \bar{b}_1\oplus \bar{b}_2 )
\beta}-1}{\beta}
=\frac{\bar{a}_1\oplus \bar{a}_2\oplus \bar{b}_1\oplus \bar{b}_2}{2}
-\beta\frac{(\bar{a}_1\oplus \bar{a}_2\oplus \bar{b}_1\oplus \bar{b}_2)^2}{8}+\cdots
$.

$\G^{C}_{2,s_0}(a,b)=
\bar{a}_1\oplus \bar{a}_2\oplus \bar{b}_1\oplus \bar{b}_2$,
$\G^{C}_{3,s_0}(a,b)=
\bar{a}_1\oplus \bar{a}_2\oplus \bar{a}_3\oplus \bar{b}_1\oplus \bar{b}_2\oplus \bar{b}_3$.

$\G^{D}_{3, s_{\hat{1}}}(a,b)=
\frac{\sqrt{1+ (\bar{a}_1\oplus \bar{a}_2\oplus \bar{a}_3
\oplus \bar{b}_1\oplus \bar{b}_2\oplus \bar{b}_3 )
\beta}-1}{\beta}$.
\end{exam}

\begin{prop}The following holds for $X=B,C,D$ and $s_i\in I_{X_n}$:
$$\pi^{(a)}_i { G}^X_n(a,b)={ G}^X_n(a,b)(u_i-\beta),$$
$$\pi^{(b)}_i { G}^X_n(a,b)=(u_i-\beta){ G}^X_n(a,b).$$
\end{prop}

\begin{proof}
  These are Proposition 3 with change of variables.
\end{proof}

\section{Identification with Schubert class}

\subsection{Equivariant $K$-theory}

Torus $T$-quivariant $K$-theory $K_{T}(X)$ of smooth algebraic variety $X$ acted by $T$ is defined as follows.
Let $Coh_{T}(X)$ be the abelian category of $T$-equivariant coherent sheaves on $X$, and
$K_{T}(X)$ be its Grothendieck group.
As we assumed $X$ to be smooth, we can give 
$K_{T}(X)$ a ring structure by defining product
comming from the tensor product of
$T$-equivariant vector bundles.
The class $[\O_{X}]$
of the structure sheaf of $X$ is the identity and
for each closed $T$-subvariety $Z\subset X$
we can associate its $T$-equivariant class $[\O_Z]\in K_{T}(X)$.
In particular
the $K$-theory Schubert class $[\mathcal O_{X^w}]$of 
the structure sheaf
$\mathcal O_{X^w}$ of the (opposite) Schubert variety $X^w=\overline{B_{-}w B/B}\subset X=G/B$, 
where $B_{-}$ is the opposite Borel subgroup, i.e.
a unique Borel subgroup with the property 
that the intersection
$B\cap B_{-}=T$ is the maximal torus contained in $B$.

For a torus $T=(\C^{*})^n$ of rank $n$,  we have
$K_T(pt)=\Z[e^{\pm t_1},\ldots,e^{\pm t_n}]$.
The Littlewood-Richardson coefficient
$c_{u,v}^w\in K_T(pt)$ is the structure constant  of $K_T(X)$
with respect to the Schubert basis $\{[\O_{X^w}]\}_{w\in W}$
defined by 
$$
[\O_{X^u}] [\O_{X^v}]=
\sum_{w\in W} c_{u,v}^w [\O_{X^w}].$$

\subsection{Algebraic localization map} We first define algebraic localization map.
This is a $K$-theoretic analogue of the universal localization map
constructed in \cite{IMN},
and extend the  (maximal) Grassmanian case of \cite{IN} to
the full flag case.
This is a $\beta$-deformation (or connective $K$-theory version)
of Lam-Shilling-Shimozono construction using $K$-NilHecke algebra.
(But in our case we must treat infinite rank Kac-Moody Lie group coresponding to root system of type $X_\infty$, for $X=A,B,C,D$.)

Let $\R^{a}_\beta:=\Z[\beta][a_1,a_2,\ldots]$ and  
$\R^{b,\bar{b}}:=\Z[\beta][b_1,\bar{b}_1,b_2,\bar{b}_2,\ldots]$. 
($\R^{b,\bar{b}}$ will play the role of $K_{T}(pt)$
(when $\beta=-1$)
 for $T=\prod_{i=1}^\infty (\C^{*})$ (as we are considering thick Schubert variety).

Let
$$P^{A}_\infty:=\R^{a}_\beta\otimes_{\Z[\beta]} \R^{b,\bar{b}},\;
P^{B}_\infty=P^{D}_\infty:=P^{A}_\infty\otimes_{\Z[\beta]} \bSS_\beta(x).$$

\noindent
For type $C$, let 
$\bSS^{C}_\beta(x)=\displaystyle\lim_{\ln}  SS^{C}_\beta(x_1,\ldots,x_n)$ and define
$$P^{C}_\infty:= P^{A}_\infty\otimes_{\Z[\beta]} \bSS^{C}_\beta(x) .$$

\noindent
For $X=A,B,C,D$,
we define $\R^{b,\bar{b}}$-linear (algebraic) localization map 
$$\Phi^X: P^{X}_\infty\to {\rm Fun}(W(X), \R^{b,\bar{b}}),$$
as follows.

\noindent
For $X=A$, $v=[v(1),v(2),\ldots]\in W(A)$ and $f(a,b,\bar{b})\in P^{A}_\infty$,
we define
$$\Phi^A(f(a,b,\bar{b}))(v):=f(v(\bar{b}),b,\bar{b}),$$
 which mean that
$f(v(\bar{b}),b,\bar{b})$ is obtained from $f(a,b,\bar{b})$ by
substituting each $a_i$ with $\bar{b}_{v(i)}$. 

\noindent
For $X=B,C,D$, $v=[v(1),v(2),\ldots]\in W(X)$ and $f(a,b,\bar{b};x)\in P^{X}_\infty$,
we define
$$\Phi^X(f(a,b,\bar{b};x))(v):=f(v(\bar{b}),b,\bar{b}; v[\bar{b}]),$$
which mean that
$f(v(\bar{b}),b,\bar{b};v[\bar{b}])$ is obtained from $f(a,b,\bar{b};x)$ by
substituting $a_i=\bar{b}_{v(i)}$ for all $i$, 
and substituting $x_i=b_{v(i)}$ if $v(i)<0$ and $x_i=0$ if $v(i)>0$.
Here we have used the convention that $b_{-i}=\bar{b}_i$.
These are $K$-theoretic analogue of the universal localization map
in \cite{IMN} \S6.1.
Let $\Delta_X$ be the root system of type $X=A,B,C,D$.

\begin{defn}(GKM subspace)
We define the Goresky-Kottwitz-MacPherson subspace 
(GKM subspace for short)  
${\rm GKM}^{X}\subset {\rm Fun}(W(X), \R^{b,\bar{b}}),$ 
as follows
$${\rm GKM}^{X}:=\left\{ f\in  {\rm Fun}(W(X), \R^{b,\bar{b}})
\;\middle|\;
\begin{array}{l}
{f(v)-f(s_\alpha v)}\in {\alpha(b)}\R^{b,\bar{b}}\\
\text{ for all } \alpha\in \Delta_X ,v\in W(X)
\end{array}
\right\}
$$
Here we write 
$s_\alpha=w s_i w^{-1}$
,$\alpha(b):=w (\alpha_i(b))$ if the root $\alpha\in \Delta_X$
has the form $\alpha=w(\alpha_i)$.

\end{defn}

\begin{prop}
The image of $\Phi^X$ has GKM property,
i.e. ${\rm Im}\; \Phi^X\subset {\rm GKM}^X$.
\end{prop}

\begin{proof}
For type $A$ case it is easy.
For type $B,C,D$ case this is a consequence of supersymmetricity
of $\SS(x)$ and $\SS^C(x)$.

\end{proof}

\begin{rem}
Actually we can show that the $\R^{b,\bar{b}}$-lienar map 
$$\widetilde{\Phi}^X:\prod_{w\in W(X)} \R^{b,\bar{b}} \G^X_w \to {GKM}^X$$
defined by
$\widetilde{\Phi}^X (\prod_{w\in W(X)} c_w \G^X_w):=\sum_{w\in W(X)} c_w \Phi^X(\G^X_w)$
is (well defined) and an isomorphism as the same  reasoning  in \cite{LSS} Proposition 2.6.
The ring $\prod_{w\in W(X)} \R^{b,\bar{b}} \G^X_w$ contains
$\frac{1}{1+\beta{a}_i}=1-\beta a_i +\beta^2 a_i^2-\beta^3 a_i^3+\cdots$ as well as
$\frac{1}{1+\beta \G^{X}_w}$.
\end{rem}

\begin{prop}
If $f\in {\rm Im}(\Phi^X)$ then $\pi_i (f)\in {\rm Im}(\Phi^X)$.
\end{prop}

We define (left) divided difference oparator $\pi_i$ on 
${\rm GKM}^X\subset {\rm Fun}(W(X), \R^{b,\bar{b}})$
as follows. (cf. \cite{IN} \S5.2.)
For $f\in {\rm Im}(\Phi^X)$,
$$(\pi_i(f))(v)=\frac{f(v)-(1+\beta \alpha_i(b))s_i^{(b)}(f(s_iv))}{\alpha_i(b)}.$$

By the GKM property of ${\rm Im}(\Phi^X)$ we have $(\pi_i(f))(v)\in \R^{b,\bar{b}}$.

\begin{prop}
$\Phi^X$ is compatible with $\pi^{(b)}_i$ and $\pi_i$, i.e.
$$\Phi^X \pi_i^{(b)}=\pi_i \Phi^X.$$

\end{prop}

$K$-theory Schubert classes are determined by the 
localization (Prop. 2.10 in \cite{LSS}),
and they are determined uniquely by
`left hand' recurrence
(Remark 2.3 in \cite{LSS} ).

\begin{prop}
(connective) $K$-theory Schubert classes $(\psi^{w})_{w\in W(X)}$,
$\psi^{w}\in {\rm Fun}(W(X), \R^{b,\bar{b}})$ are uniquely
 determined by 
\begin{itemize}

\item[(i)]
$\psi^{w}(e)=\delta_{w,e}$\\

\item[(ii)] for $v>s_i v$,\\
$\psi^{w}(v)=
\begin{cases}
s_i^{(b)} \psi^w(s_i v) & \text{ if } s_i w>w\\
(1+\beta \alpha_i({b}))s_i^{(b)}\psi^{w}(s_i v)+\alpha_i({b}) s_i^{(b)} \psi^{s_i w}(s_i v)& \text{ if } s_i w<w.\\
\end{cases}
$
\end{itemize}
\end{prop}

\begin{thm} 
For $X=A,B,C,D$,
$\left(\psi^w=\Phi^X(\G^X_w)\right)_{w\in W(X)}$ 
satisfies the recurrence relations in Proposition 10 and gives the system of 
(equivariant) Schubert classes.
\end{thm}
\begin{proof}
We use left recurrence relations.
\begin{center}
$
\G^{X}_{e}=1$ and
$
\pi^{(b)}_{i} \G^{X}_{w}=
\begin{cases}
\G^{X}_{s_{i}w}&\text{ if } s_{i}w<w\\ 
-\beta\G^{X}_{w}&\text{ if } s_{i}w>w.
\end{cases}$
\end{center}

We will write $G_w|_v:=\Phi^X(\G_w^X)(v)$.
$\psi^{w}_\beta(v):=\Phi^X(G_w^X)(v)$
(i) If we localize the generating function at $v=e$ 
we will specialize $a_i=\bar{b}_i$ and $x_i=0$ for all $i\geq 1$.
This gives the result $\G_w|_e=\delta_{w,e}$.

(ii)
By the definition of divided difference $\pi_i^{(b)}$,
we have
$$(\pi_i^{(b)} G_w)|_v=
\frac{G_w|_v-(1+\beta \alpha_i(b)) s_i^{(b)} (G_w|_{s_i v})}{\alpha_i(b)}.$$
If $s_i w>w$ then 
$$\frac{G_w|_v-(1+\beta \alpha_i(b)) s_i^{(b)} (G_w|_{s_i v})}{\alpha_i(b)}=(-\beta)G_{w}|_v.$$
From this we get
$G_w|_v=s_i^{(b)} (G_w|_{s_i v}).$\\
If $s_i w<w$ then 
$$\frac{G_w|_v-(1+\beta \alpha_i(b)) s_i^{(b)} (G_w|_{s_i v})}{\alpha_i(b)}=G_{s_i w}|_v.$$
From this we get
$$G_w|_v=(1+\beta \alpha_i(b)) s_i^{(b)}(G_w|_{s_i v})+\alpha_i(b)G_{s_i w}|_v.$$

\end{proof}
\begin{cor}
Assume
$$\G^X_u(a,b;x)\:\G^X_v(a,b;x)
=\displaystyle\sum_{w\in W(X)} 
c^{w,X}_{u,v}(\beta)\:\G^X_w(a,b;x),\;
 c^{w,X}_{u,v}(\beta)\in \R^{b,\bar{b}}
.$$
Then 
$c^{w,X}_{u,v}(\beta)|_{\beta=-1}$ is the generalized Littlewood-Richardson
coefficient $c^{w}_{u,v}$ for equivariant $K$-theory of type $X$.
($b_i$ is considered as $1-e^{t_i}$.) 
\end{cor}

\begin{rem}
 $c^{w}_{u,v}(0)$
 is the generalized Littlewood-Richardson
coefficient for equivariant cohomology
if we replace $b_i$ to $ -t_i$.
(cf. \cite{IMN}.)
\end{rem}

\begin{exam} The following is an example of the expansion.
$$\G^{C}_{s_0}(a,b;x)\:\G^{C}_{s_0}(a,b;x)=
(b_1\op b_1)\G^{C}_{s_{0}}(a,b;x)
+
\G^{C}_{s_1 s_{0}}(a,b;x)
+\beta \G^{C}_{s_0s_1s_0}(a,b;x).
$$
\end{exam}

\subsection{Explicit localization formula}
Let $\R^{b,\bar{b}}\# \Z[W]$ denote the smash product of $\R^{b,\bar{b}}$ and group algebra $\Z[W]$.
In this ring  we have $(f\otimes v)(g\otimes w)=f v^{(b)}(g) \otimes vw$.
We define $\R^{b,\bar{b}}$-linear map $\epsilon: \R^{b,\bar{b}}\# \Z[W] \to \R^{b,\bar{b}}$ by $\epsilon (f\otimes w)=f$.

\begin{prop}
Let $w,v\in W(X)$
and $v=s_{i_1} s_{i_2}\cdots s_{i_r}$ be any reduced 
decomposition of $v$ and set ${\bf i}=(i_1,i_2,\ldots,i_r)$.
For $c=(c_1,c_2,\ldots,c_r)\in \{0,1\}^r$ let
$|c|=\sum_{i=1}^{r} c_i$. 

\noindent
Then
$$\Phi^X(\G^{X}_w)(v)=
\epsilon \left(\sum_{c\in C({\bf i}, w)} \beta^{|c|-\ell(w)}\prod_{k=1}^{r}
\begin{cases}
\alpha^X_{i_k}(b) \otimes s_{i_k}&\text{ if } c_k=1\\[0.3cm]
1\otimes s_{i_k}&\text{ if } c_k=0\\
\end{cases}
\right)
,$$
where $C({\bf i}, w):=\left\{
c=(c_1,c_2,\ldots,c_r)\in \{0,1\}^r \mid
\displaystyle\prod_{k, b_k=1} u_{i_k}=\beta^{|c|-\ell(w)} u_w
\right\}.
$
\end{prop}

\begin{proof}
We can follow the proof in \cite{LSS}
Proposition 2.10.
(By induction on $\ell(v)$ and $\ell(w)$,
using left recurrence relations (i),(ii) in Proposition 10.)

\end{proof}
\begin{cor}(Vanishing property)
For $w,v\in W(X)$ we have
$$\Phi^X(\G^{X}_w)(v)=0 \text{ if } w\not \leq v.$$
$$\Phi^X(\psi^{(a)}_v\G^{X}_w)(e)=\delta_{w,v}=
\Phi^X(\psi^{(b)}_{v^{-1}}\G^{X}_w)(e).$$
\end{cor}

\section{Adjoint  polynomials}

We can also define the adjoint polynomials ${\cal H}^{X}_{n,w}$, for each $w\in W^X_n$,
( when $\beta=-1$) 
corresponding to the class of ideal sheaf $\mathcal O_{X^w}(-\partial X^w)$ 
of boundary $\partial X^w$ in $X^w$.
cf. \cite{GK,L}.
Then the pairing
$\langle\cdot\;,\cdot\rangle:K_T(X)\otimes_{R(T)} K_T(X)\to R(T)$
is given (cf. \cite{GK}) by
$$\langle v_1, v_2\rangle=\chi(X, v_1\otimes v_2)\hspace{0.5cm}
\text{ where }\hspace{0.5cm}
\chi(X,{\cal F})=\sum_{p\geq 0}(-1)^p {\rm ch}\; H^p(X,\cal F).$$
Here ${\rm ch}\; M$ for $T$-module $M$ is the formal character
defined by
$$
{\rm ch}\; M=\sum_{e^\lambda\in X(T)} \dim (M_\lambda) e^\lambda
$$
where $M_\lambda$ is the weight space corresponding to 
the weight $\lambda$.
With these notations we have (cf. \cite{GK} Proposition 2.1)
$$\langle [\mathcal O_{X_w}], [\mathcal O_{X^v}(-\partial X^v)]
 \rangle=\delta_{w,v},$$
 where $X_w=\overline{BwB/B}\subset G/B$ is the usual Schubert
 variety.
The relation between $[\mathcal O_{X^w}]$ and 
$[\mathcal O_{X^w}(-\partial X^w)]$ is as follows.
 (cf. \cite{GK} Lemma 4.2)
$$[\mathcal O_{X^w}(-\partial X^w)]=\sum_{w\leq v\leq w_0}
(-1)^{\ell(v)-\ell(w)}[\mathcal O_{X^v}].$$
We (formally) define the relative adjoint polynomial
${\cal H}^{X}_{w,v}$ for $w\leq v$ by
${\cal H}^{X}_{w,v}:=\psi^{(a)}_{w^{-1}v}(\G_{v}^{X})$.
The adjoint polynomial for $w\in W(X_n)$ is defined by
${\cal H}^{X}_{n, w}:={\cal H}^{X}_{w,w^{(n)}_0}$,
where
$w^{(n)}_0$ is the longest element in $W(X_n)$ (cf. \cite{L}).

\begin{prop} For $w\in W(X_n)$, we have
$$
{\cal H}^{X}_{n, w}=\sum_{w\leq v\leq w_0^{(n)}} \beta^{\ell(v)-\ell(w)} \G_{v}^{X}.
$$
Therefore if we specialize $\beta=-1$,
${\cal H}^{X}_{n, w}$ represents the boundary class $[\mathcal O_{X^w}(-\partial X^w)]$.
\end{prop}

\begin{proof}
We can use the property of divided difference that
$$
\psi^{(a)}_{w}=\sum_{v\leq w} \beta^{\ell(w)-\ell(v)}\pi^{(a)}_{v}.
$$

\end{proof}

\noindent
These polynomials are no longer stable but have similar properties as
Grothendieck polynomials.

\begin{prop} For $w\in W(X_n)$, we have

$$
\begin{array}{ccl}
{\cal H}^{B}_{n, e}
&=&\prod_{1\leq i\leq n-1}(1+\beta a_i)^{n-i}
\prod_{1\leq i\leq n-1}(1+\beta b_i)^{n-i}
\prod_{1\leq i\leq n}(1+\beta x_i)^{2n-1}
\\[3mm]
{\cal H}^{C}_{n, e}
&=&\prod_{1\leq i\leq n-1}(1+\beta a_i)^{n-i}
\prod_{1\leq i\leq n-1}(1+\beta b_i)^{n-i}
\prod_{1\leq i\leq n}(1+\beta x_i)^{2n}
\\[3mm]
{\cal H}^{D}_{n, e}
&=&\prod_{1\leq i\leq n-1}(1+\beta a_i)^{n-i}
\prod_{1\leq i\leq n-1}(1+\beta b_i)^{n-i}
\prod_{1\leq i\leq n}(1+\beta x_i)^{2n-2}
\end{array}
$$
\noindent
and
$${\cal H}^{X}_{n, w}
=(-1)^{\ell(w)}{\cal H}^{X}_{n, e}
\overline{\G^{X}_{n, w}},$$
\noindent
where
$\overline{\G^{X}_{n, w}}=\G^{X}_{n, w}(\overline{a},\overline{b};\overline{x})$.

\end{prop}

We can derive these formula using generating functions.
Let us define $H^{X}_n(a,b;x)$ as
$$H^{X}_n(a,b;x):=\sum_{w\in W(X_n)} (-1)^{\ell(w)} \H^{X}_{n,w}(a,b;x) u_w.$$
Then we get the following formula.

\begin{prop}
The generating function $H^{X}_n(a,b;x)$  has the following factorization.
$$H^{X}_n(a,b;x)=\H^{X}_{n,e} G^{X}_n(\bar{a},\bar{b};\bar{x}).$$
\end{prop}

Actually we can show the following property.
\begin{prop} For $s_i\in I_n^X$ we have
$$\pi^{(a)}_i H^{X}_n(a,b;x)=H^{X}_n(a,b;x) (-u_i)$$
$$\pi^{(b)}_i H^{X}_n(a,b;x)=(-u_i) H^{X}_n(a,b;x). $$
\end{prop}

\begin{prop}
(Interpolation formula)
For $F\in \bSS_\beta(x)\otimes_{\Z[\beta]} 
\R_\beta^{(a)}\otimes_{\Z[\beta]}
\R_\beta^{(b)}$,
$$F=
\sum_{v\in W(X)}
(\psi_v^{(a)}(F)|_e)\; \G^{X}_v(a,b;x)$$
\noindent
where the summation is infinite in general and
$|_e$ means the localization at $e$, i.e.
take substitutions $a_i=\bar{b}_i$ and $x_i=0$ for all $i$.
\end{prop}

\begin{proof}
$F$ can be expanded as a formal sum
$F=\sum_{v\in W(X)} c_v(F) \G^{X}_v(a,b;x)$. 
To find $c_v(F)\in \R_\beta^{(b)}$,
we can use the vanishing property  (Corollary 3), i.e.
$$\psi_v^{(a)}( \G^{X}_w(a,b;x)) |_e=\delta_{w,v}.$$
\noindent
Using the formula in the proof of Proposition 3, it follows
that $\psi_v^{(a)}( \G^{X}_{n}(a,b;x))=\G^{X}_{n}(a,b;x) u_v$ for 
$v\in W(X_n)$.
By the localization property $( \G^{X}_v(a,b;x)) |_e=\delta_{v,e}$,
we have
$\psi_w^{(a)}( \G^{X}_v(a,b;x)) |_e=\delta_{w,v}$.
From this we get the formula.
\end{proof}

\begin{cor} The equivariant Littlewood-Richardson coefficient
can be written as
$$c_{u,v}^{w,X}(\beta)=\psi^{(a)}_{w} (\G^{X}_u(a,b;x)\G^{X}_v(a,b;x)) |_e.$$
\end{cor}

\begin{thm} We have the following change of parameter formula.
$$\G_w^X(a,b;x)=
\sum_{uv=w,u\leq w}\H_{u,w}^X(\bar{c},b;0)
\G_v^X(a,c;x).$$
\end{thm}

\begin{proof}
 This is just a consequence of Proposition 16 and the definition
of $\H_{u,w}^X(a,b;x)=
\psi^{(a)}_{u^{-1} w}(\G_w^X(a,b;x))$. 
More precisely, we introduce new set of variables $c_1,c_2,\ldots$ and
$d_1,d_2,\ldots$ and consider 
$$\G_w^X(a,d;x)
\in \bSS_\beta(x)\otimes_{\Z[\beta]} 
\R_\beta^{(a)}\otimes_{\Z[\beta]}
\R_\beta^{(b)}\otimes_{\Z[\beta]}
\R_\beta^{(d)}.$$
Proposition 16 can be extended to this case by scalar extension
and
we can write
$$\G_w^X(a,d;x)=
\sum_{v\in W(X)}
(\psi_v^{(a)}(\G_w^X(a,d;x))|_e)\; \G^{X}_v(a,b;x),$$ where
$\psi_v^{(a)}(\G_w^X(a,d;x))|_e\in \R_\beta^{(d)}$
.
As 
$$\psi_v^{(a)}(\G_w^X(a,d;x))|_e=
\H_{wv^{-1},w}^X(\bar{b},d;0),$$
 we get
$$\G_w^X(a,d;x)=
\sum_{v\in W(X)}
\H_{wv^{-1},w}^X(\bar{b},d;0) \; \G^{X}_v(a,b;x).$$
Replacing $b$ by $c$ and then replacing $d$ by $b$, we get
the desired formula.
\end{proof}

\begin{rem}
There is also similar formula using second version of type $B,C,D$
double Grothendieck polynomials.
\end{rem}

\section{Combinatorial descriptions}

We give in this section two kinds of combinatorial formula
for the Grothendieck polynomials of classical types.
Actually these are essentially the same but they have 
different names and descriptions.

\subsection{Compatible sequence formula}

In \cite{FS} S. Fomin and R. Stanley
used nilCoxeter algebra to prove  compatible sequence formula
for type $A$ Schubert polynomials ${\mathfrak {S}}_{w}$ and
in \cite{FK3} S. Fomin and A.N. Kirillov gave a
compatible sequence formula
for type $A$ Grothendieck polynomials ${{\G}}_{w}$.
In \cite{BH} S. Billey and M. Haiman used Edelman-Greene type bijection
for type $B$ and $D$, to give similar combinatorial formula for Stanley symmetric functions
$E_w$ and $F_w$.
We use IdCoxeter algebra to give compatible sequence formula
for double Grothendieck polynomials $\G^A_w$,
and 
$K$-theoretic Stanley symmetric functions $\F^X_w$
for $X=B,C,D$.
To give the formula we need some notations and definitions.

We consider a sequence 
 $\tilde{a}=(\tilde{a}_1,\ldots,\tilde{a}_\ell)\in (I^X)^\ell$
 of indices of generators in $I^X$.
We denote by $\ell(\tilde{a})$ the length $\ell$ of the sequence
$\tilde{a}=(\tilde{a}_1,\ldots,\tilde{a}_\ell)$.
For type $X=B,D$, we denote by $o^B({\tilde{a}})$ the number of appearance of $0$'s in $\tilde{a}$,
by $o^D({\tilde{a}})$ the total number of appearance of $1$
and $\hat{1}$ in $\tilde{a}$.
For type $X=D$ case, we denote by
$\tilde{\tilde{a}}$  the {\it flattened} word  of $\tilde{a}=(\tilde{a}_1,\ldots,\tilde{a}_\ell)$
which is obtained from $\tilde{a}$ by replacing 
all appearence of $\hat{1}$ with $1$. cf.\cite{BH}.
For $w\in W(X)$ we define 
$$\tilde{R}(w):=\{\tilde{a}=(\tilde{a}_1,\ldots,\tilde{a}_\ell)\;|\;
u_{s_{\tilde{a}_1}}\cdots u_{s_{\tilde{a}_\ell}}=\beta^{\ell-\ell(w)} u_w
\}.$$
We define   
$$B(n;{\ell}):=\left\{
\tilde{b}=(\tilde{b}_1,\ldots,\tilde{b}_{\ell})
\;\middle|\;
\tilde{b}_i\in \Z, 
1\leq \tilde{b}_1\leq \cdots
\leq \tilde{b}_{\ell}\leq n\right\}.$$
For $\tilde{b}\in B(n;{\ell})$,
we denote by
$|\tilde{b}|$ the number of distinct $\tilde{b}_i$'s.
For $\tilde{a}=(\tilde{a}_1,\ldots,\tilde{a}_\ell)\in (I^X)^\ell$ and $\tilde{b}\in B(n;{\ell})$,
we denote by
$$\gamma(\tilde{a},\tilde{b}):=\#
\{\: i\; |\; \tilde{a}_i=\tilde{a}_{i+1}\text{ and }
\tilde{b}_i=\tilde{b}_{i+1}\}.$$

\begin{defn}
For 
$\tilde{a}=(\tilde{a}_1,\ldots,\tilde{a}_\ell)\in \tilde{R}(w)$ of $w\in W(X)$,
we define the set of compatible sequences $C^X(\tilde{a})$ as follows.

$C^{A_{n}}(\tilde{a})=
\left\{
\tilde{b}
\in B(n;{\ell})
\;\middle|\;
\begin{array}{c}
\tilde{a}_{i-1}\leq\tilde{a}_{i}
\implies
\tilde{b}_{i-1}<\tilde{b}_{i}\\
\text{ and }
\\
\tilde{b}_i\leq \tilde{a}_{i}
\end{array}
\right\}
$,

$C^{B_n}(\tilde{a})
=C^{C_n}(\tilde{a})
=
\left\{
\tilde{b}
\in B(n;{\ell})
\;\middle|
\begin{array}{c}
\tilde{a}_{i-1}\leq\tilde{a}_{i}\geq\tilde{a}_{i+1}
\implies
\tilde{b}_{i-1}<\tilde{b}_{i+1}
\end{array}
\right\}
$,

$C^{D_{n}}(\tilde{a})=
\left\{
\tilde{b}
\in B(n;\ell)
\;\middle|\;
\begin{array}{c}
\tilde{\tilde{a}}_{i-1}\leq\tilde{\tilde{a}}_{i}\geq\tilde{\tilde{a}}_{i+1}
\implies
\tilde{b}_{i-1}<\tilde{b}_{i+1}\\
\text{ and }
\\
\begin{array}{l}
\tilde{a}_i=\tilde{a}_{i+1}=1\\
\text{ or }\\
\tilde{a}_i=\tilde{a}_{i+1}=\hat{1}
\end{array}
\implies
\tilde{b}_i<\tilde{b}_{i+1}
\\
\end{array}
\right\}$.
\end{defn}

\begin{prop}(Compatible sequence formula)  cf. (
\cite{FK3},
\cite{BH} Prop 3.4 Prop 3.10)
For $w\in W(X_n)$, we have
$$
\begin{array}{ccl}
{\cal G}^A_w(x_1,\ldots, x_{n}; y_1,\ldots, y_n)&
=&
\displaystyle\sum_{\tilde{a}\in \tilde{R}(w)}
\sum_{\tilde{b}\in C^{A_n}(\tilde{a})}
\beta^{\ell(\tilde{a})-\ell(w)}
\prod_{i=1}^{\ell(\tilde{a})}
(x_{\tilde{b}_i}\oplus y_{\tilde{a}_i-\tilde{b}_i+1}),
\\[0.9cm]
{\cal F}^B_w(x_1,...,x_n)&
=&
\displaystyle\sum_{\tilde{a}\in \tilde{R}(w)}
\sum_{\tilde{b}\in C^{B_n}(\tilde{a})}
\beta^{\ell(\tilde{a})-\ell(w)}
2^{|\tilde{b}|-\gamma(\tilde{a},\tilde{b})-o^B({\tilde{a}})}
x_{\tilde{b}},
\\[0.9cm]
{\cal F}^C_w(x_1,...,x_n)&=&
\displaystyle\sum_{\tilde{a}
\in \tilde{R}(w)}
\sum_{\tilde{b}\in C^{C_n}(\tilde{a})}
\beta^{\ell(\tilde{a})-\ell(w)}
2^{|\tilde{b}|-\gamma(\tilde{a},\tilde{b})}
x_{\tilde{b}},
\\[0.9cm]
{\cal F}^D_w(x_1,...,x_n)&=&
\displaystyle\sum_{\tilde{a}\in \tilde{R}(w)}
\sum_{\tilde{b}\in C^{D_n}(\tilde{a})}
\beta^{\ell(\tilde{a})-\ell(w)}
2^{|\tilde{b}|-\gamma(\tilde{a},\tilde{b})-o^D({\tilde{a}})}
x_{\tilde{b}},\\
\end{array}
$$
\noindent
where we write
$x_{\tilde{b}}=x_{\tilde{b}_1}\cdots x_{\tilde{b}_{\ell}}$,
 $\tilde{b}=(\tilde{b}_1,\ldots,\tilde{b}_{\ell})\in B(n;{\ell})$.
\end{prop}

\begin{rem}
For cohomology case ($\beta=0$),
the formulas above for $X=C,D$ reduce to the formulas in \cite{BH} Proposition 3.4, 3.10 .
\end{rem}

\noindent

\begin{proof}
These follow immediately from the expansion of the
corresponding defining generating functions.
More precisely we will explain as follows.

The type $A_{n}$ double Grothendieck polynomials 
$\G^{A_{n}}_w(x,y)$
are defined as follows.
$$
{G}_{A_{n}}(\bar{y}_1,...,\bar {y}_{n})^{-1}
{G}_{A_{n}}(x_1,...,x_{n})
=\sum_{w\in S_{n+1}}
\G^{A_{n}}_w(x,y) u_w.
$$
By Lemma 8, the left hand side can be 
changed as below.
\begin{center}
\begin{minipage}{12cm}
$\begin{array}{lll}
\displaystyle\prod_{b=1}^{n}
\left(\prod_{a=n}^{b} h_{a}(x_b\oplus y_{a-b+1})\right)&=&
[h_n(x_1\oplus y_n)\cdots  h_1(x_1\oplus y_2)h_1(x_1\oplus y_1)]\\
&&[h_n(x_2\oplus y_{n-1})\cdots  h_2(x_2\oplus y_1)]\\
&&\hspace{1cm}\vdots\\
&&[h_n(x_n\oplus y_1)].
\end{array}
$
\end{minipage}
\end{center}

\noindent
By expanding  this it is easy to see that there is a one 
to one correspondence between the compatible sequences
and the expanded terms
which implies the first formula.

For $X=C$ case,
${F}^C_n(x_1,...,x_n)=
F^C_n(x_1)\cdots F^C_n(x_n)$.
As\\
$F^C_n(x_b)=h_{n-1}(x_b)h_{n-2}(x_b)\cdots h_1(x_b) h_{0}(x_b) 
h_{0}(x_b)h_1(x_b)\cdots h_{n-2}(x_b) h_{n-1}(x_b)\\
=\left(\displaystyle\prod_{a=n-1}^{0} h_{a}(x_b)\right)
\left(\displaystyle\prod_{a=0}^{n-1} h_{a}(x_b)\right)
$,
each  expanded term has the form
$x_b^{m} u_{a_1} u_{a_2}\cdots u_{a_m}$ with two cases below.

\noindent
(1) there exists $i$ , $1\leq i \leq m$ such that
$$n> a_1>a_2> \cdots >a_{i-1}> a_i <a_{i+1}< \cdots < a_{m}< n$$
(if $i=1$ then we assume $a_{i-1}=n$ and
if $i=m$ then we assume $a_{i+1}=n$),
or

\noindent
(2) there exists $i$ , $1\leq i \leq m-1$ such that
$$n> a_1>a_2> \cdots >a_{i-1}> a_i =a_{i+1}< \cdots < a_{m}< n.$$
(if $i=1$ then we assume $a_{i-1}=n$).

\noindent
There are two possible choices of $u_{a_i}$ for each (1) case while
there is only one possibility for each case (2),
which explains the factor
$2^{|\tilde{b}|-\gamma(\tilde{a},\tilde{b})}$.

For type $B$ case, (1) has the exception of $a_i=0$ in which case
 it can occur only once. This explains the factor of 
$2^{|\tilde{b}|-\gamma(\tilde{a},\tilde{b})-o^B({\tilde{a}})}$.

For the case of type $D$ is similar.
For each (1) case, if $a_i=1$ or $a_i=\hat{1}$ it corresponds to the case
$\tilde{a}_i=1$ or
$\tilde{a}_i=\hat{1}$ (the flattened $\tilde{\tilde{a}}_i=1$)
of compatible sequence.
The factor $2^{1-0-1}$
counts correctly in each of this case.
The case (2) will be modified when
$a_i=1$ and $a_{i+1}=\hat{1}$.
It corresponds to either
$\tilde{a}_i=1\text{ and }\tilde{a}_{i+1}=\hat{1}$
or
$\tilde{a}_i=\hat{1}\text{ and }\tilde{a}_{i+1}={1}$
(the flattened word $\tilde{\tilde{a}}_i=\tilde{\tilde{a}}_{i+1}=1$)
.
The corresponding factor $2^{1-0-2}$ appears twice which 
sum up to
1.
So the factor 
$2^{|\tilde{b}|-\gamma(\tilde{a},\tilde{b})-o^D({\tilde{a}})}$
properly counts the expanded terms.
\end{proof}

\begin{exam}

Type $D$, $n=2$ case 
$w=[\bar{1},\bar{2}]=
s_1 s_{\hat{1}}$.

$\tilde{b}=(1,1)$ is a compatible sequence for
$\tilde{a}=(1,\hat{1}),(\hat{1},1)$.

$\tilde{b}=(1,2)$ is a compatible sequence for
$\tilde{a}=(1,\hat{1}),(\hat{1},1)$.

$\tilde{b}=(2,2)$ is a compatible sequence for
$\tilde{a}=(1,\hat{1}),(\hat{1},1)$.

$\tilde{b}=(1,1,2)$ is a compatible sequence for
$\tilde{a}=(1,\hat{1},1),(\hat{1},1,\hat{1}),(1,\hat{1},\hat{1}),(\hat{1},1,1)$.

$\tilde{b}=(1,2,2)$ is a compatible sequence for
$\tilde{a}=(1,\hat{1},1),(\hat{1},1,\hat{1}),(1,1,\hat{1}),(\hat{1},\hat{1},1)$.

$\tilde{b}=(1,1,2,2)$ is a compatible sequence for\\
\hspace{6cm}
{$\tilde{a}=(1,\hat{1},1,\hat{1}),(\hat{1},1,\hat{1},1),
(\hat{1},1,1,\hat{1}),(1,\hat{1},\hat{1},1)$.}

\noindent
There are no other compatible sequences and
the sum of the terms becomes 

$\F^{D}_{s_1 s_{\hat{1}}}(x_1,x_2)
= x_1^2+2 x_1 x_2+x_2^2+2\beta x_1^2 x_2+2\beta x_1 x_2^2+
\beta^2 x_1^2 x_2^2=(x_1\op x_2)^2$.

\end{exam}

\begin{rem}
We can use Proposition 2 and 17 to get the  
compatible sequence formula
for double Grothendieck polynomials of type $B,C,D$.
\end{rem}

\subsection{Pipe dream (extended EYD) formula}

There is a state sum formula 
called pipe dream formula
for type $A$ Schubert polynomials.
It was first introduced in \cite{BB} and called RC-graph .
They use two kinds of configurations to realize  a pattern 
of given permutation
$w\in S_n$ and the sum of weights for each pattern gives the
Schubert polynomial $S_w$. Example 6 gives such
a pattern, where we use 
\FGC and \FGP to realize a pattern of the permutation
$w=[3,1,4,2]$ (we delete useless lines).
The name of pipe dream is given after the name of similar game.
Here we connect $i$ to $w(i)$ ($1\leq i\leq n$) for a given $w\in S_n$.
We can extend this to type $B,C,D$ cases as follows.

Let us recall the type $A$ pipe dream formula for double Grothendieck
polynomials.
Fix a sequence $\Delta^{A}_{n-1}$ of simple reflections which gives a reduced expression for the longest element $w_0$ in $S_n$ as follows.
$$\Delta^{A}_{n-1}:=(s_{n-1} |s_{n-2},s_{n-1}|\cdots|s_1,s_2,\cdots,s_{n-1})
= (d_1, d_2, \cdots, d_N),$$
where $N:=n(n-1)/2$
, i.e. 
if $ \frac{m(m-1)}{2}< k\leq \frac{m(m+1)}{2}$ then 
$d_k=s_{n-1-(k-\frac{m(m+1)}{2})}$
. We arrange this sequence in triangular form,
 with coordinate $(i,j)$  for $i+j<n$,
from left to right from top to bottom
as follows.

$\begin{array}{ccccc}
d_1\\
d_2&d_3\\
d_4&d_5&d_6\\&\cdots&&\\
d_{M+1}&d_{M+2}&d_{M+3}&\cdots&d_N\\
\end{array}
$
$\begin{array}{ccccc}
(1,n-1)\\
(1,n-2)&(2,n-2)\\
(1,n-3)&(2,n-3)&(3,n-3)\\
&\cdots\\
(1,1)&(2,1)&\cdots&(n-1,1)\\
\end{array}
$

\noindent
where $M=\frac{(n-1)(n-2)}{2}$.
Set $wt(d_k)=a_i\oplus b_j$ when $d_k$ is in the coordinate $(i,j)$,
cf. Example 6 for $n=4$ case. 

Each term in the expansion of right  hand side of
\begin{eqnarray}
{G}_{A_{n-1}}(\bar{b}_1,...,\bar{b}_{n-1})^{-1}
{G}_{A_{n-1}}(a_1,...,a_{n-1})=
\prod_{j=n-1}^{1}
\left(\prod_{i=1}^{n-j} h_{i+j-1}(a_{i}\oplus b_{j})\right)
\end{eqnarray}
corresponds to a subsequence of $\Delta^{A}_{n-1}$.
To give the statement uniformly for $X=A,B,C,D$, we need some
notations in general.
\begin{defn}
Given a sequence 
$\Delta=(d_1,d_2,\ldots,d_N)$ of simple reflections in $W(X)$
and
an element $w\in W(X)$ with length $\ell(w)=\ell$, let 
${\rm Rsub}(\Delta, w)$ be the set of 
subsequences of $\Delta$ each element of which gives 
a reduced expression of $w$ with length $\ell=\ell(w)$.
i.e.
$${\rm Rsub}(\Delta, w):=
\{
(d_{j_1},d_{j_2},\ldots,d_{j_\ell})\;|\;
1\leq j_1<j_2<\cdots<j_\ell\leq N, d_{j_1} d_{j_2} \cdots d_{j_\ell}=w\}.$$
We will call $\D\in {\rm Rsub}(\Delta, w)$ an extended EYD.
We also define 
the set $B(\D)$ of backward movable positions
for
an element 
$\D=(d_{j_1},d_{j_2},\ldots,d_{j_\ell})\in {\rm Rsub}(\Delta,w)$, 
by considering  $j_{\ell+1}:=N+1$,
$$
B(\D):=\{d_j \;|\; j\leq N, \exists p \text{ such that } {j_p}<j<{j_{p+1}},
d_{j_1}d_{j_2} \cdots d_{j_p}=(d_{j_1}d_{j_2} \cdots d_{j_p})\star d_j
 \}.  
$$

\end{defn}
 
 \noindent
Then we have the following extended EYD 
formula.

\begin{thm} For $w\in S_n$, we have
$$
\G^{A_{n-1}}_w(a,b)=\sum_{\D\in {\rm RSub}(\Delta^{A}_{n-1},w)} Wt(\D),
$$
\noindent
where
$$Wt(\D)=\prod_{\Box\in \D} wt(\Box)\times
\prod_{\bigcirc\in B(\D)}(1+\beta wt(\bigcirc)).$$
\end{thm}
\begin{proof}
This is just a consequence of the equation (2).
\end{proof}
There is a one to one correspondence
between ${\rm Rsub}(\Delta^{A}_{n-1},w)$
and the set of reduced pipe dreams  $PD(w)$ for $w\in S_n$.
For $\D\in {\rm Rsub}(\Delta^{A}_{n-1},w)$,
we put 
two patterns on $\Delta^{A}_{n-1}$.
One is 
\setlength\unitlength{0.17truemm}
\begin{picture}(30,30)(0,0)
\multiput(0,0)(0,30){2}{\line(1,0){30}}
\multiput(0,0)(30,0){2}{\line(0,1){30}}
\put(15,32){\line(0,-1){34}}
\put(32,15){\line(-1,0){34}}
\end{picture}
which corresponds to the selected box 
\setlength\unitlength{0.17truemm}
\begin{picture}(30,30)(0,0)
\multiput(0,0)(0,30){2}{\line(1,0){30}}
\multiput(0,0)(30,0){2}{\line(0,1){30}}
\put(15,15){\makebox(0,0){\bf \Large$\Box$}}
\end{picture}
in EYD configuration of $\D$.
The other case (corresponding to unselected box
\setlength\unitlength{0.17truemm}
\begin{picture}(30,30)(0,0)
\multiput(0,0)(0,30){2}{\line(1,0){30}}
\multiput(0,0)(30,0){2}{\line(0,1){30}}
\end{picture}
) we put
\setlength\unitlength{0.17truemm}
\begin{picture}(30,30)(0,0)
\multiput(0,0)(0,30){2}{\line(1,0){30}}
\multiput(0,0)(30,0){2}{\line(0,1){30}}
\put(13,32){\line(1,-1){22}}
\put(16,-1){\line(-1,1){22}}
\end{picture}
in the box.
Each selected box 
\setlength\unitlength{0.17truemm}
\begin{picture}(30,30)(0,0)
\multiput(0,0)(0,30){2}{\line(1,0){30}}
\multiput(0,0)(30,0){2}{\line(0,1){30}}
\put(15,15){\makebox(0,0){\bf \Large$\Box$}}
\end{picture}
 in $\D$ corresponds to a word of
the reduced expression of $w$ corresponding to $\D$, which appears
in a subsequence of $\Delta^{A}_{n-1}$. 
The positions of the elements in $B(\D)$ are indicated by
circles 
\setlength\unitlength{0.17truemm}
\begin{picture}(30,30)(0,0)
\multiput(0,0)(0,30){2}{\line(1,0){30}}
\multiput(0,0)(30,0){2}{\line(0,1){30}}
\put(15,15){\makebox(0,0){$\bigcirc$}}
\end{picture}
 in the examples below (cf. \cite{IN}).

\begin{exam}
 Type $A_3$.

\hspace{-1cm}
\setlength\unitlength{0.25truemm}
\begin{picture}(130,150)(-50,0)

\put(-45,60){$\Delta^{A}_3=$}

\multiput(0,90)(30,-30){4}{\line(1,0){30}}
\multiput(90,0)(-30,30){4}{\line(0,1){30}}

\multiput(0,60)(30,-30){3}{\line(1,0){30}}
\multiput(60,0)(-30,30){3}{\line(0,1){30}}

\multiput(0,30)(30,-30){2}{\line(1,0){30}}
\multiput(30,0)(-30,30){2}{\line(0,1){30}}

\multiput(0,0)(30,-30){1}{\line(1,0){30}}
\multiput(0,0)(-30,30){1}{\line(0,1){30}}

\put(15,75){\makebox(0,0){$d_1$}}
\put(15,45){\makebox(0,0){$d_2$}}
\put(45,45){\makebox(0,0){$d_3$}}
\put(15,15){\makebox(0,0){$d_4$}}
\put(45,15){\makebox(0,0){$d_5$}}
\put(75,15){\makebox(0,0){$d_6$}}
\end{picture}\hspace{0cm}
\begin{picture}(150,150)(-50,0)

\put(-45,60){$=$}

\multiput(0,90)(30,-30){4}{\line(1,0){30}}
\multiput(90,0)(-30,30){4}{\line(0,1){30}}

\multiput(0,60)(30,-30){3}{\line(1,0){30}}
\multiput(60,0)(-30,30){3}{\line(0,1){30}}

\multiput(0,30)(30,-30){2}{\line(1,0){30}}
\multiput(30,0)(-30,30){2}{\line(0,1){30}}

\multiput(0,0)(30,-30){1}{\line(1,0){30}}
\multiput(0,0)(-30,30){1}{\line(0,1){30}}

\multiput(75,15)(-30,30){3}{\makebox(0,0){$s_3$}}
\multiput(45,15)(-30,30){2}{\makebox(0,0){$s_2$}}
\multiput(15,15)(-30,30){1}{\makebox(0,0){$s_1$}}

\end{picture}\hspace{0cm}
\begin{picture}(150,150)(-50,0)

\put(-45,60){$wt=$}
\multiput(0,90)(30,-30){4}{\line(1,0){30}}
\multiput(90,0)(-30,30){4}{\line(0,1){30}}

\multiput(0,60)(30,-30){3}{\line(1,0){30}}
\multiput(60,0)(-30,30){3}{\line(0,1){30}}

\multiput(0,30)(30,-30){2}{\line(1,0){30}}
\multiput(30,0)(-30,30){2}{\line(0,1){30}}

\multiput(0,0)(30,-30){1}{\line(1,0){30}}
\multiput(0,0)(-30,30){1}{\line(0,1){30}}

\put(15,45){\makebox(0,0){\tiny $[1;2]$}}
\put(45,45){\makebox(0,0){\tiny $[2;2]$}}
\put(15,15){\makebox(0,0){\tiny$[1;1]$}}

\put(75,15){\makebox(0,0){\tiny$[3;1]$}}

\put(15,75){\makebox(0,0){\tiny$[1;3]$}}
\put(45,15){\makebox(0,0){\tiny$[2;1]$}}

\put(20,105){$[i;j]=a_i\op b_j$}

\end{picture}

\noindent
\hspace{0cm}
\begin{minipage}{2.8cm}
Example\\
 of  an  EYD\\
configuration\\
for
 $w=[3,1,4,2]\\
\hspace{0.8cm} =s_2 s_3 s_1$.
 \vspace{2cm}

\end{minipage}
\begin{picture}(150,150)(-50,0)

\put(-45,60){$\D=$}

\multiput(0,90)(30,-30){4}{\line(1,0){30}}
\multiput(90,0)(-30,30){4}{\line(0,1){30}}

\multiput(0,60)(30,-30){3}{\line(1,0){30}}
\multiput(60,0)(-30,30){3}{\line(0,1){30}}

\multiput(0,30)(30,-30){2}{\line(1,0){30}}
\multiput(30,0)(-30,30){2}{\line(0,1){30}}

\multiput(0,0)(30,-30){1}{\line(1,0){30}}
\multiput(0,0)(-30,30){1}{\line(0,1){30}}

\multiput(75,15)(-30,30){3}{\makebox(0,0){$3$}}
\multiput(45,15)(-30,30){2}{\makebox(0,0){$2$}}
\multiput(15,15)(-30,30){1}{\makebox(0,0){$1$}}

\put(45,45){\makebox(0,0){\huge$\Box$}}
\put(15,45){\makebox(0,0){\huge$\Box$}}
\put(15,15){\makebox(0,0){\huge$\Box$}}

\put(75,15){\makebox(0,0){\large$\bigcirc$}}

\put(-60,-35)
{$\D=(d_2,d_3,d_4)=(s_2,s_3,s_1)\in {\rm Rsub}(\Delta^A_3,w),
B(\D)=\{d_6\}$}

\put(-60,-55)
{$Wt(\D)=(a_1\op b_2)(a_2\op b_2)(a_1\op b_1)(1+\beta\; a_3\op b_1)$}

\end{picture}
\hspace{0cm}
\begin{picture}(150,150)(-50,0)

\put(40,105){pipe dream} 
\put(40,90){diagram for $\D$}

\multiput(0,90)(30,-30){4}{\line(1,0){30}}
\multiput(90,0)(-30,30){4}{\line(0,1){30}}

\multiput(0,60)(30,-30){3}{\line(1,0){30}}
\multiput(60,0)(-30,30){3}{\line(0,1){30}}

\multiput(0,30)(30,-30){2}{\line(1,0){30}}
\multiput(30,0)(-30,30){2}{\line(0,1){30}}

\multiput(0,0)(30,-30){1}{\line(1,0){30}}
\multiput(0,0)(-30,30){1}{\line(0,1){30}}

\put(45,60){\line(0,-1){30}}
\put(30,45){\line(1,0){30}}

\put(15,30){\line(0,-1){30}}
\put(30,15){\vector(-1,0){30}}

\put(15,60){\line(0,-1){30}}
\put(30,45){\vector(-1,0){30}}

\put(75,0){\line(-1,1){30}}
\put(105,0){\line(-1,1){45}}
\put(45,0){\line(-1,1){15}}
\put(15,60){\vector(-1,1){15}}
\put(45,60){\vector(-1,1){45}}

\multiput(15,0)(30,0){4}{\makebox(0,0){$\bullet$}}

\put(15,-15){\makebox(0,0){1}}
\put(45,-15){\makebox(0,0){2}}
\put(75,-15){\makebox(0,0){3}}
\put(105,-15){\makebox(0,0){4}}

\put(-10,15){\makebox(0,0){1}}
\put(-10,45){\makebox(0,0){2}}
\put(-10,75){\makebox(0,0){3}}
\put(-10,105){\makebox(0,0){4}}

\put(100,50){$\in {\rm PD}(w)$}
\end{picture}

\end{exam}

\begin{exam}
Type $A_3$,
$w=[1,4,3,2]=s_3 s_2 s_3=s_2 s_3 s_2\in S_4$.

\noindent
One can show that 
$\G^{A_{3}}_w(a,b)|_{a=1,b=0}=5+5\beta+\beta^2$.

\begin{picture}(150,150)(-50,0)

\put(20,105){EYD1}

\put(-50,60){$\D_1=$}

\multiput(0,90)(30,-30){4}{\line(1,0){30}}
\multiput(90,0)(-30,30){4}{\line(0,1){30}}

\multiput(0,60)(30,-30){3}{\line(1,0){30}}
\multiput(60,0)(-30,30){3}{\line(0,1){30}}

\multiput(0,30)(30,-30){2}{\line(1,0){30}}
\multiput(30,0)(-30,30){2}{\line(0,1){30}}

\multiput(0,0)(30,-30){1}{\line(1,0){30}}
\multiput(0,0)(-30,30){1}{\line(0,1){30}}

\multiput(75,15)(-30,30){3}{\makebox(0,0){3}}
\multiput(45,15)(-30,30){2}{\makebox(0,0){2}}
\multiput(15,15)(-30,30){1}{\makebox(0,0){1}}

\put(45,45){\makebox(0,0){\Large$\Box$}}
\put(45,15){\makebox(0,0){\Large$\Box$}}
\put(75,15){\makebox(0,0){\Large$\Box$}}

\end{picture}
\begin{picture}(150,150)(-50,0)

\put(20,105){EYD2}

\put(-50,60){$\D_2=$}

\multiput(0,90)(30,-30){4}{\line(1,0){30}}
\multiput(90,0)(-30,30){4}{\line(0,1){30}}

\multiput(0,60)(30,-30){3}{\line(1,0){30}}
\multiput(60,0)(-30,30){3}{\line(0,1){30}}

\multiput(0,30)(30,-30){2}{\line(1,0){30}}
\multiput(30,0)(-30,30){2}{\line(0,1){30}}

\multiput(0,0)(30,-30){1}{\line(1,0){30}}
\multiput(0,0)(-30,30){1}{\line(0,1){30}}

\multiput(75,15)(-30,30){3}{\makebox(0,0){3}}
\multiput(45,15)(-30,30){2}{\makebox(0,0){2}}
\multiput(15,15)(-30,30){1}{\makebox(0,0){1}}

\put(15,75){\makebox(0,0){\Large$\Box$}}
\put(75,15){\makebox(0,0){\Large$\Box$}}
\put(45,15){\makebox(0,0){\Large$\Box$}}

\put(45,45){\makebox(0,0){\large$\bigcirc$}}

\end{picture}
\begin{picture}(150,150)(-50,0)

\put(20,105){EYD3}

\put(-50,60){$\D_3=$}

\multiput(0,90)(30,-30){4}{\line(1,0){30}}
\multiput(90,0)(-30,30){4}{\line(0,1){30}}

\multiput(0,60)(30,-30){3}{\line(1,0){30}}
\multiput(60,0)(-30,30){3}{\line(0,1){30}}

\multiput(0,30)(30,-30){2}{\line(1,0){30}}
\multiput(30,0)(-30,30){2}{\line(0,1){30}}

\multiput(0,0)(30,-30){1}{\line(1,0){30}}
\multiput(0,0)(-30,30){1}{\line(0,1){30}}

\multiput(75,15)(-30,30){3}{\makebox(0,0){3}}
\multiput(45,15)(-30,30){2}{\makebox(0,0){2}}
\multiput(15,15)(-30,30){1}{\makebox(0,0){1}}

\put(15,45){\makebox(0,0){\Large$\Box$}}
\put(15,75){\makebox(0,0){\Large$\Box$}}
\put(75,15){\makebox(0,0){\Large$\Box$}}

\put(45,15){\makebox(0,0){\large$\bigcirc$}}

\end{picture}

\begin{picture}(150,150)(-50,0)

\put(20,105){EYD4}

\put(-50,60){$\D_4=$}

\multiput(0,90)(30,-30){4}{\line(1,0){30}}
\multiput(90,0)(-30,30){4}{\line(0,1){30}}

\multiput(0,60)(30,-30){3}{\line(1,0){30}}
\multiput(60,0)(-30,30){3}{\line(0,1){30}}

\multiput(0,30)(30,-30){2}{\line(1,0){30}}
\multiput(30,0)(-30,30){2}{\line(0,1){30}}

\multiput(0,0)(30,-30){1}{\line(1,0){30}}
\multiput(0,0)(-30,30){1}{\line(0,1){30}}

\multiput(75,15)(-30,30){3}{\makebox(0,0){3}}
\multiput(45,15)(-30,30){2}{\makebox(0,0){2}}
\multiput(15,15)(-30,30){1}{\makebox(0,0){1}}

\put(15,45){\makebox(0,0){\Large$\Box$}}
\put(15,75){\makebox(0,0){\Large$\Box$}}
\put(45,45){\makebox(0,0){\Large$\Box$}}

\put(45,15){\makebox(0,0){\large$\bigcirc$}}
\put(75,15){\makebox(0,0){\large$\bigcirc$}}

\end{picture}
\begin{picture}(150,150)(-50,0)

\put(20,105){EYD5}

\put(-50,60){$\D_5=$}

\multiput(0,90)(30,-30){4}{\line(1,0){30}}
\multiput(90,0)(-30,30){4}{\line(0,1){30}}

\multiput(0,60)(30,-30){3}{\line(1,0){30}}
\multiput(60,0)(-30,30){3}{\line(0,1){30}}

\multiput(0,30)(30,-30){2}{\line(1,0){30}}
\multiput(30,0)(-30,30){2}{\line(0,1){30}}

\multiput(0,0)(30,-30){1}{\line(1,0){30}}
\multiput(0,0)(-30,30){1}{\line(0,1){30}}

\multiput(75,15)(-30,30){3}{\makebox(0,0){3}}
\multiput(45,15)(-30,30){2}{\makebox(0,0){2}}
\multiput(15,15)(-30,30){1}{\makebox(0,0){1}}

\put(45,45){\makebox(0,0){\Large$\Box$}}
\put(15,45){\makebox(0,0){\Large$\Box$}}
\put(45,15){\makebox(0,0){\Large$\Box$}}

\put(75,15){\makebox(0,0){\large$\bigcirc$}}

\end{picture}
\begin{picture}(150,150)(-50,0)
\put(-30,40){
$wt=
\begin{array}{cccc}
a_1\op b_3\\[0.2cm]
a_1\op b_2&a_2\op b_2\\[0.2cm]
a_1\op b_1& a_2\op b_1&a_3\op b_1\\
\end{array}
$}
\end{picture}
\\


\begin{minipage}{12cm}
$\begin{array}{ccl}
Wt(\D_1)&=&(a_2\op b_2)(a_2\op b_1)(a_3\op b_1)
\\
Wt(\D_2)&=&(a_1\op b_3)(a_2\op b_1)(a_3\op b_1)
(1+\beta (a_2\op b_2))\\
Wt(\D_3)&=&(a_1\op b_3)(a_1\op b_2)(a_3\op b_1)
(1+\beta (a_2\op b_1))\\
Wt(\D_4)&=&(a_1\op b_3)(a_1\op b_2)(a_2\op b_2)
(1+\beta (a_2\op b_1))(1+\beta(a_3\op b_1))\\
Wt(\D_5)&=&(a_1\op b_2)(a_2\op b_2)(a_2\op b_1)
(1+\beta (a_3\op b_1))\\
\end{array}
$
\end{minipage}

\noindent
From these data we get

$
\G^{A_{3}}_{s_2 s_3 s_2}(a,b)=Wt(\D_1)+Wt(\D_2)+Wt(\D_3)+Wt(\D_4)+Wt(\D_5)
$.

\end{exam}

\begin{rem}
Actually there is an algorithm to create all the 
extended EYD diagrams for a given $w\in 
{{S}}_{n}$.
The algorithm is essentially written in \cite{BB}.
Combinatorics related to extended EYD diagrams
(including type $B,C,D$ case)
  will be 
discussed elsewhere.
\end{rem}


\begin{lem}
For type $B_n$ or $C_n$ case,
 we can rewrite the generating function 
 $\G_n^{X}(a,b;x)=\displaystyle\sum_{w\in W(X_n)} \G^X_{n,w}(a,b;x) u_w$ in
 Definition 8
 as follows.
\begin{eqnarray}\hspace{-1cm}
\left(\displaystyle\prod_{j=n-1}^{1}\prod_{i=1}^{n-j} h_{i+j-1}(x_{n-i+1}\oplus b_{j})\right)
\left(\displaystyle\prod_{i=n}^{1}\prod_{j=n}^{i} h_{j-i}(x^X_{i,j})\right)
\left(\displaystyle\prod_{i=n-1}^{1}\prod_{j=1}^{n-i} h_{i+j-1}(x_{i}\oplus a_{j})\right)
\end{eqnarray}
\noindent
where $x^X_{i,j}=x_i\oplus x_j$ if $i \neq j$ ,
$x^B_{i,i}=x_i$
and
$x^C_{i,i}=x_i\oplus x_i$
.
\end{lem}
\begin{proof}
We will prove for the case of $C_n$.
By Lemma 2 and Lemma 3, we have\\
\noindent
$\begin{array}{lll}
&&F_n(x_1)F_n(x_2)\\
&=&[n-1,1]_{x_1}[0]_{x_1\op x_1}[1,n-1]_{x_1}
[n-1,1]_{x_2}[0]_{x_2\op x_2}[1,n-1]_{x_2}\\
&=&[n-1,1]_{x_1}[0]_{x_1\op x_1} [n-1,1]_{x_2} [1,n-1]_{x_1} [0]_{x_2\op x_2}[1,n-1]_{x_2}\\
&=&[n-1,1]_{x_1}[n-1,2]_{x_2} [0]_{x_1\op x_1} [1]_{x_2}[1]_{x_1} [0]_{x_2\op x_2}[2,n-1]_{x_1}[1,n-1]_{x_2}\\
&=&[n-1,1]_{x_1}[n-1,2]_{x_2}  [0]_{x_1\op x_1} [1]_{x_1\op x_2} [0]_{x_2\op x_2}[2,n-1]_{x_1}[1,n-1]_{x_2}.\\
\end{array}$

Continuing this procedure we get

\noindent
$\begin{array}{lll}
&&F_n(x_1)F_n(x_2)\cdots F(x_n)\\
&=&[n-1,1]_{x_1}[n-1,2]_{x_2} \cdots [n-1]_{x_{n-1}} \\
&&\hspace{3cm}\nabla(x_1,\ldots,x_n)[n-1]_{x_2}[n-2,n-1]_{x_3}\cdots [1,n-1]_{x_n}\\
&=&{G}_{A_{n-1}}(x_1,...,x_{n-1})\nabla(x_1,\ldots,x_n) {G}_{A_{n-1}}(\bar {x}_{n},...,\bar{x}_{2})^{-1},
\end{array}
$

where
$$\nabla(x_1,\ldots,x_n) =\prod_{i=1}^n \left(\prod_{j=i}^{n} h_{i+j-1}(x_i\op x_j)\right).$$
Then reversing the order of $x_1,\ldots,x_n$ to $x_n,\ldots,x_1$  and  using Lemma 8, we get
$\G_n^{C}(a,b;x)=$\\
\hspace{1cm}${G}^{A}_{n-1}(x_{n},\ldots,x_{2},b_1,\ldots,b_{n-1})
\nabla(x_n,\ldots,x_1)
{G}^{A}_{n-1}(a_1,\ldots,a_{n-1},x_{1},\ldots,x_{n-1}).
$
This is the formula (3).
For type $B_n$ case almost the same argument holds.


\end{proof}

For $X=B,C$,
we define a sequence of simple reflections of $W_n(X)$ as follows.
$$\Delta^X_n=
\Delta^{A}_{n-1}
(s_0,s_1,s_2,\cdots, s_{n-1})^n=(d_1,d_2,\ldots,d_N),$$
where $N=\frac{n(n-1)}{2}+n^2$. 
Note that in this case it doesn't correspond to a reduced decomposition of the longest element $w^{X_n}_0$.
We arrange this sequence
in trapezoidal form from left to right
and from top to bottom,
 cf. Example 8 for the type $C$, $n=3$ case.
The coordinate of $d_k$ is as follows.
For $1\leq k\leq \frac{n(n-1)}{2}$,
the coordinate of $d_k$ is $(k-\frac{m(m-1)}{2}, 2n-m)$
if $\frac{m(m-1)}{2}<k\leq \frac{m(m+1)}{2}$.
For $\frac{n(n-1)}{2}<k\leq N$,
the coordinate of $d_k$ is $(s+t ,n-s)$
if $k=\frac{n(n-1)}{2}+s n+t$, 
for some integers $s, t$ with $0\leq s$ and $1\leq  t\leq n$.

We also define the weight 
$wt^C_n( d_k)=p_i\op q_j$ if the coordinate of
$d_k$ is $(i,j)$,
where $p_i$ and $q_j$ are defined as follows.
\begin{eqnarray}
p_i=x_{n+1-i}\text{ if }1\leq i\leq n\text{ and }
p_i=a_{i-n}\text{ if } n< i,\\
q_j=x_j\text{ if }1\leq j\leq n\text{ and }
q_j=b_{j-n}\text{ if }n<j.
\end{eqnarray}
\noindent
For type $X=B$ case, we set
$wt^B_n( d_k)=wt^C_n( d_k)$ except for the case of
the coordinate of $d_k$ is $(i,n+1-i)$ in which case
we set $wt^B_n( d_k)=x_{n+1-i}$.

\begin{thm}
For $w\in W_n(X),X=B,C$, we have
$$
\G^{X}_{n,w}(a,b;x)=\sum_{\D\in {\rm RSub}(\Delta^X_n,w)} Wt^{X}_n(\D),
$$
\noindent
where
$$Wt^X_n(\D)=\prod_{\Box\in \D} wt^X_n(\Box)\times
\prod_{\bigcirc\in B(\D)}(1+\beta wt^X_n(\bigcirc)).$$

\end{thm}

\begin{proof}
This follows from the equation (3). Indeed 
using the relation $h_{i}(x) h_{j}(y)=h_{j}(y) h_{i}(x)$ for $i,j\geq 0$ s.t. $|i-j|>1$,
(3) can be rewritten
as follows, from which we get the result.
$$
G^C_n(a,b;x)=\left(\displaystyle\prod_{j=2n-1}^{n+1}\prod_{i=1}^{2n-j} h_{i+j-n-1}(p_{i}\oplus q_{j})\right)
\left(\displaystyle\prod_{j=n}^{1}\prod_{i=n+1-j}^{2n-j} 
h_{i+j-n-1}(p_i\op q_j)
\right).
$$
$$
G^B_n(a,b;x)=\left(\displaystyle\prod_{j=2n-1}^{n+1}\prod_{i=1}^{2n-j} h_{i+j-n-1}(p_i\op q_j)\right)
\left(\displaystyle\prod_{j=n}^{1}
\prod_{i=n+1-j}^{2n-j} 
h_{i+j-n-1}(wt^B_n(i,j))
\right),
$$
where $wt^B_n(i,j)=p_i\op q_j$ if $i+j>n+1$ and
$wt^B_n(i,n+1-i)=q_{n+1-i}$ for $1\leq i\leq n$.

\end{proof}

\vspace{0.5cm}

\begin{exam}

Type $C_3$, $w=[2,\bar{3},1]=s_2s_1 s_2 s_0 s_1$.

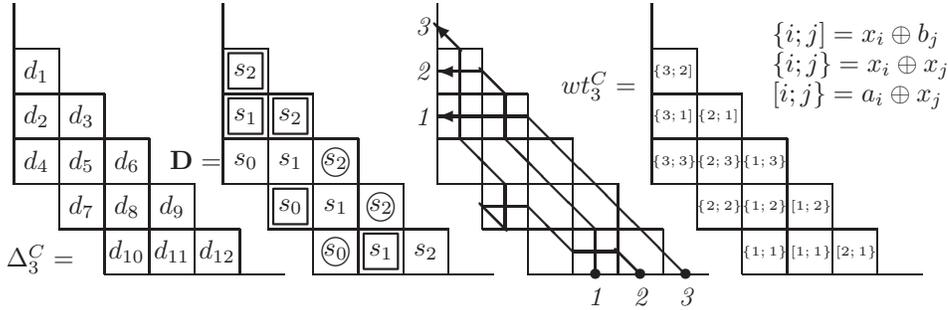
\setlength\unitlength{0.20truemm}
\begin{picture}(120,200)
\put(-5,5){$\Delta^{C}_{3}=$}

\multiput(0,150)(30,-30){6}{\line(1,0){30}}
\multiput(150,0)(-30,30){6}{\line(0,1){30}}

\multiput(0,120)(30,-30){5}{\line(1,0){30}}
\multiput(120,0)(-30,30){5}{\line(0,1){30}}

\multiput(0,90)(30,-30){4}{\line(1,0){30}}
\multiput(90,0)(-30,30){4}{\line(0,1){30}}

\multiput(0,60)(30,-30){3}{\line(1,0){30}}
\multiput(60,0)(-30,30){3}{\line(0,1){30}}

\put(15,135){\makebox(0,0){$d_1$}}
\put(15,105){\makebox(0,0){$d_2$}}
\put(45,105){\makebox(0,0){$d_3$}}
\put(15,75){\makebox(0,0){$d_4$}}
\put(45,75){\makebox(0,0){$d_5$}}
\put(75,75){\makebox(0,0){$d_6$}}
\put(45,45){\makebox(0,0){$d_7$}}
\put(75,45){\makebox(0,0){$d_8$}}
\put(105,45){\makebox(0,0){$d_9$}}
\put(75,15){\makebox(0,0){$d_{10}$}}
\put(105,15){\makebox(0,0){$d_{11}$}}
\put(135,15){\makebox(0,0){$d_{12}$}}

\end{picture}
\hspace{0cm}
\setlength\unitlength{0.20truemm}
\begin{picture}(130,200)

\put(-35,70){$\D=$}

\multiput(0,150)(30,-30){6}{\line(1,0){30}}
\multiput(150,0)(-30,30){6}{\line(0,1){30}}

\multiput(0,120)(30,-30){5}{\line(1,0){30}}
\multiput(120,0)(-30,30){5}{\line(0,1){30}}

\multiput(0,90)(30,-30){4}{\line(1,0){30}}
\multiput(90,0)(-30,30){4}{\line(0,1){30}}

\multiput(0,60)(30,-30){3}{\line(1,0){30}}
\multiput(60,0)(-30,30){3}{\line(0,1){30}}

\multiput(135,15)(-30,30){5}{\makebox(0,0){$s_2$}}
\multiput(105,15)(-30,30){4}{\makebox(0,0){$s_1$}}
\multiput(75,15)(-30,30){3}{\makebox(0,0){$s_0$}}

\put(15,135){\makebox(0,0){\huge$\Box$}}

\put(45,45){\makebox(0,0){\huge$\Box$}}
\put(105,15){\makebox(0,0){\huge$\Box$}}
\put(15,105){\makebox(0,0){\huge$\Box$}}
\put(45,105){\makebox(0,0){\huge$\Box$}}

\put(75,75){\makebox(0,0){\large$\bigcirc$}}
\put(105,45){\makebox(0,0){\large$\bigcirc$}}
\put(75,15){\makebox(0,0){\large$\bigcirc$}}

\end{picture}
\setlength\unitlength{0.20truemm}
\begin{picture}(130,200)
\multiput(0,150)(30,-30){6}{\line(1,0){30}}
\multiput(150,0)(-30,30){6}{\line(0,1){30}}

\multiput(0,120)(30,-30){5}{\line(1,0){30}}
\multiput(120,0)(-30,30){5}{\line(0,1){30}}
\multiput(0,90)(30,-30){4}{\line(1,0){30}}
\multiput(90,0)(-30,30){4}{\line(0,1){30}}

\multiput(0,60)(30,-30){3}{\line(1,0){30}}
\multiput(60,0)(-30,30){3}{\line(0,1){30}}
\thicklines

\put(15,150){\line(0,-1){30}}
\put(30,135){\vector(-1,0){30}}

\put(45,60){\line(0,-1){30}}
\put(30,45){\line(1,0){30}}

\put(105,30){\line(0,-1){30}}
\put(90,15){\line(1,0){30}}

\put(15,120){\line(0,-1){30}}
\put(30,105){\vector(-1,0){30}}

\put(45,120){\line(0,-1){30}}
\put(30,105){\line(1,0){30}}

\thicklines
\put(46,119){\line(-1,1){18}}
\put(46,29){\line(-1,1){18}}
\put(16,149){\vector(-1,1){18}}
\put(136,-1){\line(-1,1){18}}

\put(165,0){\line(-1,1){105}}

\put(105,30){\line(-1,1){60}}
\put(90,15){\line(-1,1){30}}
\put(45,60){\line(-1,1){30}}

\put(105,-15){\makebox(0,0){1}}
\put(135,-15){\makebox(0,0){2}}
\put(165,-15){\makebox(0,0){3}}
\put(-10,105){\makebox(0,0){1}}
\put(-10,135){\makebox(0,0){2}}
\put(-10,165){\makebox(0,0){3}}

\multiput(105,0)(30,0){3}{\makebox(0,0){$\bullet$}}
\end{picture}
\setlength\unitlength{0.20truemm}
\begin{picture}(120,200)
\put(-60,120){$wt^C_3=$}

\multiput(0,150)(30,-30){6}{\line(1,0){30}}
\multiput(150,0)(-30,30){6}{\line(0,1){30}}

\multiput(0,120)(30,-30){5}{\line(1,0){30}}
\multiput(120,0)(-30,30){5}{\line(0,1){30}}

\multiput(0,90)(30,-30){4}{\line(1,0){30}}
\multiput(90,0)(-30,30){4}{\line(0,1){30}}

\multiput(0,60)(30,-30){3}{\line(1,0){30}}
\multiput(60,0)(-30,30){3}{\line(0,1){30}}

\put(15,75){\makebox(0,0){\tiny$\{3;3\}$}}
\put(45,75){\makebox(0,0){\tiny$\{2;3\}$}}
\put(75,15){\makebox(0,0){\tiny$\{1;1\}$}}
\put(75,45){\makebox(0,0){\tiny$\{1;2\}$}}
\put(75,75){\makebox(0,0){\tiny$\{1;3\}$}}

\put(14,135){\makebox(0,0){\tiny$\{3;2]$}}
\put(14,105){\makebox(0,0){\tiny$\{3;1]$}}
\put(44,105){\makebox(0,0){\tiny$\{2;1]$}}

\put(106,15){\makebox(0,0){\tiny$[1;1\}$}}
\put(106,45){\makebox(0,0){\tiny$[1;2\}$}}
\put(136,15){\makebox(0,0){\tiny$[2;1\}$}}

\put(45,45){\makebox(0,0){\tiny$\{2;2\}$}}

\put(80,155){$\{i;j]=x_i\op b_j$}
\put(80,135){$\{i;j\}=x_i\op x_j$}
\put(80,115){$[i;j\}=a_i\op x_j$}
\end{picture}

\vspace{0.5cm}
$\D=(d_1,d_2,d_3,d_7,d_{11})=(s_2,s_1,s_2,s_0,s_1)\in {\rm Rsub}(\Delta^{C}_3,w)$,

$B(\D)=\{d_6, d_9, d_{10}\}$,

$Wt_3^C(\D)=
(x_3\op b_2)
(x_3\op b_1)
(x_2\op b_1)
(x_2\op x_2)
(x_1\op a_1)\\
\hspace{3cm}\times
(1+\beta (x_1\op x_3))
(1+\beta (a_1\op x_2))
(1+\beta (x_1\op x_1))
$.

\end{exam}

\vspace{0.5cm}

Comparing this to the type $A$ case, we get the following formula.

\begin{prop}
For $w\in W(A_{n-1})\subset W(B_n)=W(C_n)$, we have

\hspace{-1cm}
$
\G^{B}_{n,w}(a,b;x)=
\G^{C}_{n,w}(a,b;x)=
\G^{A_{2n-1}}_{1^{n}\times w}(x_1,\ldots,x_n,a_1,\ldots,a_{n-1}\:,\:
x_1,\ldots,x_n,b_1,\ldots,b_{n-1}).
$
\end{prop}
\begin{proof}
According to the setting of weight $wt^X(d_k)$, it is clear that $\G^{B}_{n,w}(a,b;x)=
\G^{C}_{n,w}(a,b;x)$ for $w\in W(A_{n-1})$.
Comparing  the weights of type $A_{2n-1}$ and $C_{n}$ cases
with  the formula 
(4) and (5)
, we get
$$\G^{C}_{n,w}(a,b;x)=
\G^{A_{2n-1}}_{1^{n}\times w}(x_n,\ldots,x_1,a_1,\ldots,a_{n-1}\:,\:
x_1,\ldots,x_n,b_1,\ldots,b_{n-1}).$$
But in this case the first $n$ variables of $\G^{A_{2n-1}}_{1^{n}\times w}
(a_1,\ldots,a_{2n-1}, b_1,\ldots b_{2n-1})$
 are commutative, because
 {$(1^{n}\times w) s_i>(1^{n}\times w)$} for
$1\leq i\leq n-1$
 .
\end{proof}

For type $D_n$ case, \underline{we assume $n=2m$ an even integer}. For odd $n=2m-1$ case
we can get the formula by just erasing the last variable $x_{2m}=0$ for $n=2m$ case. 

\begin{lem}
The generating function 
$$G^D_{n}(a,b; x)=
{G}_{A_{n-1}}(\bar{b}_1,...,\bar {b}_{n-1})^{-1}
{F}_n^D(x)
{G}_{A_{n-1}}(a_1,...,a_{n-1})$$ can be rewritten as follows 
by using Yang-Baxter relations.

\begin{eqnarray}
\hspace{-1cm}
\left(\displaystyle\prod_{j=n-1}^{1}\prod_{i=1}^{n-j} h_{i+j-1}(x_{n-i+1}\oplus b_{j})\right)
\left(\prod_{i=n-1}^{1}
\displaystyle\prod_{j=n}^{i+1} h_{i,j}(x_{i,j})\right)
\left(\displaystyle\prod_{i=n-1}^{1}
\displaystyle\prod_{j=1}^{n-i} h_{i+j-1}(x_{i}\oplus a_{j})\right)
\end{eqnarray}
\noindent
where
$h_{i,j}(x_{i,j})
:=
h_{j-i}(x_i\oplus x_j)$ if $j-i\geq 2$, 
$h_{i,i+1}(x_{i,i+1})
:=
h_{\hat{1}}(x_i\oplus x_{i+1})$ if $i=$odd
and
$h_{i,i+1}(x_{i,i+1})
:=
h_{{1}}(x_i\oplus x_{i+1})$ if $i=$even.
\end{lem}

\begin{proof}
The argument is almost the same as Lamma 9 and we omit the details
\end{proof}

\vspace{0.5cm}

Let us define (for $n=2m$ case ) the sequence $\Delta^{D}_n$ of simple reflections by
$$\Delta^{D}_n:=\Delta^{A}_{n-1} ((s_{\hat{1}}, s_2,\ldots, s_{n-1})
(s_1, s_2,\ldots s_{n-1}) )^{m}=
(d_1,d_2,\ldots, d_N).$$
where $N=\frac{n(n-1)}{2}+(n-1)^n$.
We arrange these in trapezoidal form with coordinate
begining from
$(1,2n-1)$ to $(2n-1,1)$ as type $C_n$ case, but
skip the coordinate $(i, n+1-i)$ for $1\leq i\leq n$.
The weight is $wt^D_n(d_k)=p_i\oplus q_j$ when $d_k$ is in the
coordinate$(i,j)$.
Formally it is the same as type $C_n$ case but
we skip the position 
 $(i, n+1-i)$, $1\leq i\leq n$ for the type $D_n$ case.
See Example 9 below for $n=4$ case.

\begin{thm}
For $w\in W_n(D)$ $(n=2m)$, we have
$$
\G^{D}_{n,w}(a,b;x)=\sum_{{\bf D}\in {\rm RSub}(\Delta^D_n,w)} Wt^{D}_n({\D}),
$$
\noindent
where
$$Wt^D_n({\D})=\prod_{\Box\in {\D}} wt^D_n(\Box)\times
\prod_{\bigcirc\in B({\D})}(1+\beta wt^D_n(\bigcirc)).$$
For $n=2m-1$ case we can use the above formula with  $x_{2m}=0$.
\end{thm}
\begin{proof}
This follows by expanding  the product (6), which
can be rewritten using Yang-Baxter relations as in the proof of Theorem 5 as follows.
$$G^D_n(a,b;x)=
\left(\displaystyle\prod_{j=2n-1}^{n+1}\prod_{i=1}^{2n-j} 
h_{i+j-1}(p_{i}\oplus q_{j})\right)
\left(\prod_{j=n}^{1}
\displaystyle\prod_{i=n+2-j}^{2n-j} h_{i,j}(p_i\oplus q_{j})
\right)
$$
where
$p_i$, $q_j$ are defined by (4), (5), and 

\begin{center}
$h_{i,j}=
\left\{
\begin{array}{ll}
h_{i+j-n-1}&\text{ if } i+j>n+2,\\
h_{{1}}&\text{ if } i+j=n+2\text{ and } j \text{ is odd,}\\
h_{\hat{1}}&\text{ if } i+j=n+2\text{ and } j \text{ is even.}\\
\end{array}
\right.
$
\end{center}

\end{proof}

\vspace{1cm}

\begin{exam}

Type $D_4$ , $w=[\bar{2},4,\bar{1},3]=s_{3} s_{\hat{1}} s_2$.

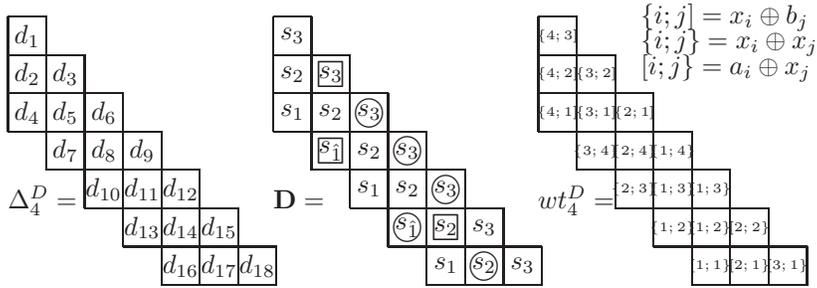
\begin{picture}(200,240)

\multiput(0,120)(30,-30){5}{\line(1,0){90}}
\multiput(0,120)(30,-30){5}{\line(0,1){90}}

\multiput(0,150)(30,-30){5}{\line(1,0){90}}
\multiput(30,120)(30,-30){5}{\line(0,1){90}}

\put(0,180){\line(1,0){60}}
\put(0,210){\line(1,0){30}}
\put(180,0){\line(0,1){60}}
\put(210,0){\line(0,1){30}}

\put(15,195){\makebox(0,0){$d_1$}}
\put(15,165){\makebox(0,0){$d_2$}}
\put(45,165){\makebox(0,0){$d_3$}}
\put(15,135){\makebox(0,0){$d_4$}}
\put(45,135){\makebox(0,0){$d_5$}}
\put(75,135){\makebox(0,0){$d_6$}}
\put(45,105){\makebox(0,0){$d_7$}}
\put(75,105){\makebox(0,0){$d_8$}}
\put(105,105){\makebox(0,0){$d_9$}}
\put(75,75){\makebox(0,0){$d_{10}$}}
\put(105,75){\makebox(0,0){$d_{11}$}}
\put(135,75){\makebox(0,0){$d_{12}$}}
\put(105,45){\makebox(0,0){$d_{13}$}}
\put(135,45){\makebox(0,0){$d_{14}$}}
\put(165,45){\makebox(0,0){$d_{15}$}}
\put(135,15){\makebox(0,0){$d_{16}$}}
\put(165,15){\makebox(0,0){$d_{17}$}}
\put(195,15){\makebox(0,0){$d_{18}$}}

\put(0,60){$\Delta^{D}_4=$}
\end{picture}
\begin{picture}(200,240)

\multiput(0,120)(30,-30){5}{\line(1,0){90}}
\multiput(0,120)(30,-30){5}{\line(0,1){90}}

\multiput(0,150)(30,-30){5}{\line(1,0){90}}
\multiput(30,120)(30,-30){5}{\line(0,1){90}}

\put(0,180){\line(1,0){60}}
\put(0,210){\line(1,0){30}}
\put(180,0){\line(0,1){60}}
\put(210,0){\line(0,1){30}}

\multiput(45,105)(60,-60){2}{\makebox(0,0){$s_{\hat{1}}$}}
\multiput(15,135)(60,-60){3}{\makebox(0,0){$s_1$}}
\multiput(15,165)(30,-30){6}{\makebox(0,0){$s_2$}}
\multiput(15,195)(30,-30){7}{\makebox(0,0){$s_3$}}

\put(45,165){\makebox(0,0){\Large$\Box$}}
\put(45,105){\makebox(0,0){\Large$\Box$}}
\put(135,45){\makebox(0,0){\Large$\Box$}}

\multiput(75,135)(30,-30){3}{\makebox(0,0){\large$\bigcirc$}}
\put(105,45){\makebox(0,0){\large$\bigcirc$}}
\put(165,15){\makebox(0,0){\large$\bigcirc$}}

\put(0,60){${\D}=$}
\end{picture}
\begin{picture}(200,240)

\put(0,60){$wt^D_4=$}
\multiput(0,120)(30,-30){5}{\line(1,0){90}}
\multiput(0,120)(30,-30){5}{\line(0,1){90}}

\multiput(0,150)(30,-30){5}{\line(1,0){90}}
\multiput(30,120)(30,-30){5}{\line(0,1){90}}

\put(0,180){\line(1,0){60}}
\put(0,210){\line(1,0){30}}
\put(180,0){\line(0,1){60}}
\put(210,0){\line(0,1){30}}

\put(135,15){\makebox(0,0){\tiny$[1;1\}$}}
\put(165,15){\makebox(0,0){\tiny$[2;1\}$}}
\put(195,15){\makebox(0,0){\tiny$[3;1\}$}}

\put(105,45){\makebox(0,0){\tiny$\{1;2\}$}}
\put(105,75){\makebox(0,0){\tiny$\{1;3\}$}}
\put(105,105){\makebox(0,0){\tiny$\{1;4\}$}}
\put(75,75){\makebox(0,0){\tiny$\{2;3\}$}}
\put(75,105){\makebox(0,0){\tiny$\{2;4\}$}}

\put(45,105){\makebox(0,0){\tiny$\{3;4\}$}}

\put(135,45){\makebox(0,0){\tiny$[1;2\}$}}
\put(135,75){\makebox(0,0){\tiny$[1;3\}$}}

\put(165,45){\makebox(0,0){\tiny$[2;2\}$}}

\put(15,135){\makebox(0,0){\tiny$\{4;1]$}}
\put(45,135){\makebox(0,0){\tiny$\{3;1]$}}
\put(75,135){\makebox(0,0){\tiny$\{2;1]$}}

\put(15,165){\makebox(0,0){\tiny$\{4;2]$}}
\put(45,165){\makebox(0,0){\tiny$\{3;2]$}}

\put(15,195){\makebox(0,0){\tiny$\{4;3]$}}

\put(80,205){$\{i;j]=x_i\op b_j$}
\put(80,185){$\{i;j\}=x_i\op x_j$}
\put(80,165){$[i;j\}=a_i\op x_j$}

\end{picture}
\vspace{0.2cm}

$\D=(d_3,d_7,d_{14})=(s_3, s_{\hat{1}}, s_2)\in {\rm RSub}(\Delta^D_4,w)$, 
$B(\D)=\{
d_6,d_9,d_{12},d_{13},d_{17}
\}$.

$Wt^D_4({\D})=
(x_3\op b_2)
(x_3\op x_4)
(a_1\op x_2)\times\\
\hspace{1cm}
(1+\beta (x_2\op b_1))
(1+\beta (x_1\op x_4))
(1+\beta (a_1\op x_3))
(1+\beta (x_1\op x_2))
(1+\beta (a_2\op x_1))
$.

\end{exam}

\begin{prop}
For $w\in  \langle s_2, s_3,\ldots,s_{n-1} \rangle \subset W(D_n)$, we have

\hspace{-1cm}
$
\G^{D}_{n,w}(a,b;x)=
\G^{A_{2n-1}}_{1^{n}\times w}(x_1,\ldots,x_n,a_1,\ldots,a_{n-1}\:,\:
x_1,\ldots,x_n,b_1,\ldots,b_{n-1}).
$
\end{prop}

\begin{rem}
If  $w$ is a maximal Grassmannian element of type $B_n,C_n$ or $D_n$,
then the above pipe dream formula can be regarded as 
the excited Young diagram formula of
\cite{IN} Th 9.2. Therefore  the above gives a generalization of the EYD formula.
Therefore we can call Theorem 4,5,6 as extended EYD formula. Note also that
even in type $A$ (Theorem 4) case the formula given in this form is a
compressed form compared to compatible sequence formula (Proposition 17).
\end{rem}

\vspace{0.5cm}

{\bf Acknowledgements.} ${}$ We thank the anonymous referee 
for carefully reading the first version of this document and
making suggestions to ameliorate the paper.
We also thank Takeshi Ikeda for usefull comments
on the first version.

\end{document}